\documentclass{amsart}
\usepackage{tikz}
\usepackage{hyperref}   
\usepackage{amssymb,xcolor}
\usepackage{stmaryrd} 
\usepackage{enumerate}

\usepackage{mathtools}  
\mathtoolsset{showonlyrefs}
\usepackage{import}
\usepackage{graphicx}
\theoremstyle{plain}
\newtheorem{theorem}{Theorem}[section]
\theoremstyle{plain}
\newtheorem{corollary}[theorem]{Corollary}
\theoremstyle{plain}
\newtheorem{definition}[theorem]{Definition}
\theoremstyle{plain}
\newtheorem{lemma}[theorem]{Lemma}
\theoremstyle{remark}
\newtheorem{remark}[theorem]{Remark}
\theoremstyle{plain}
\newtheorem{proposition}[theorem]{Proposition}
\numberwithin{equation}{section}
\numberwithin{figure}{section}
\theoremstyle{plain}

\begin{document}

\title[Scalar Curvature Rigidity]{Scalar curvature rigidity of domains in a 3-dimensional warped product}

\author{Xiaoxiang Chai}
\address{Department of Mathematics, POSTECH, Pohang, Gyeongbuk, South Korea}
\email{ xxchai@kias.re.kr, xxchai@postech.ac.kr}

\author{Gaoming Wang}
\address{Yau Mathematical Sciences Center, Tsinghua University, Beijing, 100084, China}
\email{gmwang@tsinghua.edu.cn}
\address{}


\begin{abstract}
  A warped product with a spherical factor and a logarithmically concave 
  warping function satisfies a scalar curvature rigidity of the Llarull type. We
  develop a scalar curvature rigidity of the Llarull type for a general class of domains
  in a three dimensional
  spherical warped product. In the presence of rotational symmetry, we identify this class of domains as those satisfying a boundary condition analogous to the
  logarithmic concavity of the warping function. 
\end{abstract}

{\maketitle}

\section{Introduction}

Llarull proved a scalar curvature rigidity
theorem for the standard $n$-spheres. 
\begin{theorem}[{\cite{llarull-sharp-1998}}]
  \label{llarull}Let $g$ be a smooth metric on the n-sphere with the metric
  comparison $g \geqslant \bar{g}$ and the scalar curvature comparison $R_g
  \geqslant n (n - 1)$. Then $g = \bar{g}$.
\end{theorem}
An interesting feature of this scalar
curvature rigidity for spheres comparing to that of torus
{\cite{schoen-existence-1979}}, the Euclidean space {\cite{schoen-proof-1979}}
and the hyperbolic space {\cite{minoo-scalar-1989}} is the requirement of a
metric comparison $g \geqslant \bar{g}$. A counterexample without the requirement $g \geqslant \bar{g}$ was given in \cite{brendle-scalar-2011} (for half $n$-sphere fixing the geometry of the boundary), but $g\geqslant \bar g$ can be weakened, see
{\cite{listing-scalar-arxiv-2010}}. Llarull also showed that the condition $g \geqslant \bar{g}$ can be
formulated more generally as the existence of a distance non-increasing map
$F : (M, g) \to (\mathbb{S}^n, \bar{g})$ of non-zero degree.

Recently, there were efforts in extending Llarull's theorem to a spherical
warped product
\begin{equation}
  (\bar{M}^n, \bar{g}) := ([t_-, t_+] \times S^{n - 1}, \mathrm{d} t^2 + \psi
  (t)^2 g_{\mathbb{S}^{n - 1}}) \text{ with } t_- < t_+, \text{ } (\log
  \psi)'' < 0, \label{spherical warped product std}
\end{equation}
by spinors {\cite{cecchini-scalar-2024}}, {\cite{bar-scalar-2024}},
{\cite{wang-scalar-2023}}, by $\mu$-bubbles {\cite{gromov-four-2021}},
{\cite{hu-rigidity-2023}} and by spacetime harmonic functions
{\cite{hirsch-rigid-2022}}.

We are interested in the Llarull type theorems of domains in the spherical warped
product \eqref{spherical warped product}. Although the form \eqref{spherical
warped product} can also be considered as a domain in a larger spherical
warped product, our focus will be on such domains with boundaries that are not
necessarily given by $t$-level sets. Previously, this direction has been
explored by Lott {\cite{lott-index-2021}}, Wang-Xie
{\cite{wang-dihedral-2023}} and Chai-Wan {\cite{chai-scalar-2024}}, which all
involved spinors.
Gromov first suggested the use of stable capillary minimal surface in
studying the scalar curvature rigidity of Euclidean balls (see Section 5.8.1
of {\cite{gromov-four-2021}};
\text{{\itshape{Spin}}}-\text{{\itshape{Extremality}}} \text{{\itshape{of}}}
\text{{\itshape{Doubly}}} \text{{\itshape{Punctured}}}
\text{{\itshape{Balls}}}) and Li {\cite{li-polyhedron-2020}} in
three-dimensional Euclidean dihedral rigidity. 
In this article, we make use of the stable capillary surfaces with prescribed
(varying) contact angle and prescribed mean curvature, or in 
the terminology of Gromov {\cite{gromov-four-2021}}, (part of) the boundary of a stable
capillary $\mu$-bubble.
This is also a further
development of the previous work {\cite{chai-scalar-2023-arxiv}} and the recent work of Ko-Yao \cite{ko-scalar-2024} in the Euclidean case.

We consider three dimensions, fix a metric $g_{S^2}$ of positive
Gauss curvature on the $2$-sphere $S^2$, and the following
\begin{equation}
  (\bar{M}^3, \bar{g}) := ([t_-, t_+] \times S^{2}, \mathrm{d} t^2 + \psi
  (t)^2 g_{{S}^{2}}) \text{ with } t_- < t_+, \text{ } (\log
  \psi)'' < 0. \label{spherical warped product}
\end{equation}
We reserve $g_{\mathbb{S}^2}$ for the standard round metric on the 2-sphere $S^2$.
We use $t_x$ to indicate the $t$ coordinate and $p_x$
to denote the $S^2$ coordinate of a point $x\in M \subset [t_-,t_+]\times S^2$. Let $K(p)$ be the Gauss
curvature at $p\in (S^2,g_{S^2})$.

Let $P_{\pm} = \{t_{\pm} \} \times {S}^2$, we always assume that $M$ lies between $P_\pm$ such that $P_\pm \cap \partial M$ is non-empty. If $P_\pm \cap \partial M$ contains only a point, we set this point to be $p_\pm$. Let $\partial_s M =
\partial M\backslash (P_+ \cup P_-)$ and $\bar{X}$ be the unit outward normal
of $\partial_s M$ in $M$ with respect to $\bar{g}$.
We fix 
\begin{equation}
    \bar{h}(t)=2\psi'(t)/\psi(t)
\end{equation}
and $\bar{\gamma}$ to be the dihedral angles formed by
$\partial_s M$ and $\Sigma_t = (\{t\} \times {S}^2) \cap M$ that is given 
by
\begin{equation}
  \cos \bar{\gamma} = \bar{g} (\bar{X}, \partial_t) . \label{angle}
\end{equation}
We call the $t$-suplevel set $\Omega_t^+$ and $t$-sublevel set $\Omega_t^-$.
We also need to fix some more
conventions for the direction of the unit normal, the sign of the mean
curvatures and the dihedral angles. Let $\Sigma$ be a surface with boundary on $\partial_sM$ and
separates $P_+ \cap \partial M$ and $P_- \cap \partial M$, we always fix the
direction of the unit normal $N$ of $\Sigma$ to be the direction which points
inside of the region bounded by $\Sigma,$ $P_+ \cap \partial M$ and
$\partial_s M$. The mean curvature is then the trace of the second fundamental
form $\nabla N$. We fix $\gamma_{\Sigma}$ to be the contact angle formed by $\Sigma$ and $\partial_sM$, that is,
$\cos \gamma_{\Sigma} = \langle X, N \rangle$. For the mean
curvature of $\partial_s M$, it is always computed with respect to the outward
unit normal. The geometric quantity on $(M, \bar{g})$ comes with a bar unless
otherwise specified (see Figure \ref{fig:notations}).

As is well known, the warped product metric
\eqref{spherical warped product} is conformal to a direct product metric. Indeed, let $s = \int^t \tfrac{1}{\psi (\tau)} \mathrm{d}
\tau$, then $\mathrm{d} s = \tfrac{1}{\psi (t)} \mathrm{d} t$ and
\begin{equation}
  \mathrm{d} t^2 + \psi (t)^2 g_{{S}^2} = \psi (t)^2 \mathrm{d} s^2 +
  \psi (t)^2 g_{{S}^2} = \psi (t)^2 (\mathrm{d} s^2 + g_{{S}^2})
  \label{conformally related metric}
\end{equation}
where $t = t (s)$ is implicitly given by $s = \int^t \tfrac{1}{\psi (\tau)}
\mathrm{d} \tau$.

Now we state our scalar curvature rigidity result. (For a quick feel, we refer to Theorem \ref{smooth boundary llarull}.)

\begin{theorem} \label{rigidity depending on structures}
    Assume that $\bar{g}$ is a Riemannian metric on $\bar{M}^3$ as in \eqref{spherical warped product}. Assume that $M$ is a domain in $\bar{M}$ such that $\partial_s M$ is convex with respect to the
  conformally related metric $\mathrm{d} s^2 + g_{S^2}$ where $s$ is given in
  \eqref{conformally related metric} and $g$ is a smooth metric on $M$. If $g$ satisfies the metric comparison
    \begin{equation}
        g \geqslant \bar{g},
    \end{equation}
    and the scalar curvature comparison
    \begin{equation}
        R_g \geqslant R_{\bar{g}},
    \end{equation}
    and the mean curvature comparison
    \begin{equation}
        H_{\partial_s M} \geqslant \bar{H}_{\partial_s M}
    \end{equation}
    along $\partial_s M$, and
    further that
    $P_\pm \cap M$ and $\psi(t)$ satisfy either of the following conditions: 
    \begin{enumerate}
        \item \label{smooth} $\psi(t_\pm)>0$, $P_\pm \cap \partial M$ only contains a point where $\partial M$ is smooth under both metrics $g$ and $\bar g$;
        \item \label{item euclid cone} $\psi(t_\pm)>0$, $P_\pm \cap \partial M$ only contains a point where $\partial M$ where the tangent cone of $(M,g)$ at $p_\pm$ is Euclidean, and the tangent cone of $(M,\bar g)$ at $p_\pm$ is a round (solid) circular cone and that $(|t-t_\pm|^{-1}\Sigma_t,|t-t_\pm|^{-2} \bar{g})$ converges to its axial section as $t\to t_\pm$;
        \item \label{item metric cone} $\psi(t) = a|t-t_\pm| +o (|t-t_\pm|^2)$ as $t\to t_\pm$ where $a>0$ is a constant such that the tangent cone of $(M,\bar g)$ at $p_\pm$ is of non-negative Ricci curvature and $(M,g)$ admits a metric tangent cone at $p_\pm$;
        \item \label{item simple} $\psi(t)>0$ on $[t_-,t_+]$, $P_\pm\cap \partial M$ is a disk, the mean curvatures $\pm H_{P_\pm \cap\partial M} \geqslant \pm \bar{h}_{P_\pm \cap \partial M} = \pm\bar{h} (t_\pm)$ and the dihedral angles $\pm\gamma_{P_+ \cap \partial M} \geqslant
    \pm\bar{\gamma}  |_{P_\pm \cap \partial M}$ formed by $\partial_s M$ and $P_\pm \cap
    \partial M$,
    \end{enumerate}
    then $g = \bar g$.
\end{theorem}

\begin{remark}
    The mean curvature comparisons can be reformulated as $H_{\partial M}
\geqslant \bar{H}_{\partial M}$ on $\partial M$ if all mean curvatures are
computed with respect to the outward unit normal. For convergence of sequences of Riemannian manifolds and notions of tangent cones at a point of a Riemannian manifold, we refer to
{\cite[Chapter 8]{burago-course-2001}}. We believe that the extra condition on the tangent cone of $(M,\bar g)$ at $p_\pm$ can be dropped in item (\ref{item euclid cone}).
\end{remark}

We use a special case of $(M,\bar{g})$ to illustrate the convexity of $\partial_s M$
with respect to the metric $\mathrm{d}s^2+g_{S^2}$ given in \eqref{conformally related metric}. Assume that $g_{S^2}$ is just the standard round
metric, that is,
\begin{equation} g_{S^2} = g_{\mathbb{S}^2} = \mathrm{d} r^2 + \sin^2 r \mathrm{d} \theta^2,
   \text{ } r \in [0, \pi], \text{ } \theta \in \mathbb{S}^1, \label{polar coordinates of S2}
   \end{equation}
   using the polar coordinates
and $M$ is given by
\[ M = \{(t, r, \theta) : \text{ } t \in [t_-, t_+], \text{ } r \leqslant \rho
   (t) < \tfrac{\pi}{2}, \text{ } \theta \in \mathbb{S}^1 \}  \]
   for some positive function $\rho(t)$ on $[t_-,t_+]$.
In this case, the prescribed contact angle $\bar{\gamma}$ only depends on $t$.
It is easy to check that the convexity of $\partial_s M$ with respect to
$\mathrm{d} s^2 + g_{\mathbb{S}^2}$ is equivalent to $\tfrac{\mathrm{d}
\bar{\gamma}}{\mathrm{d} s} < 0$. Also, we find that \begin{equation}
\tfrac{\mathrm{d}
\bar{\gamma}}{\mathrm{d} t} < 0\label{boundary log concavity}
\end{equation}
using \eqref{conformally related metric} or that the angles are conformally invariant. The mean curvature of a $t$-level set
is given by \[\bar{h} (t) := 2 \psi^{- 1} \tfrac{\mathrm{d} \psi}{\mathrm{d} t}
= 2 \tfrac{\mathrm{d} (\log \psi)}{\mathrm{d} t}.\] The logarithmic concavity
of $\psi$ is equivalent to the more geometric statement that the mean curvature of
the $t$-level set is monotonically decreasing as $t$ increases. The condition
$\tfrac{\mathrm{d} \bar{\gamma}}{\mathrm{d} t} < 0$ can be viewed as a
boundary analog of the logarithmic concavity. 
Geometrically, \eqref{boundary log concavity} says that the dihedral angles \eqref{angle} formed by $\Sigma_t$
and $\partial_s M$ monotonically decreases along the $\partial_t$ direction
with respect to the metric $\bar{g}$. This condition answers a question raised by
Gromov at the end of {\cite[Section 5.8.1]{gromov-four-2021}}. See also Chai-Wan {\cite[Theorem
1.1]{chai-scalar-2024}}.

We can have various scalar curvature rigidity assuming different combinations of geometric structures at $P_\pm$. The following is a simple corollary of Theorem \ref{rigidity depending on structures} where both $P_\pm \cap \partial M$ take the structure as specified in item (\ref{smooth}) of Theorem \ref{rigidity depending on structures}.

\begin{theorem} \label{smooth boundary llarull}
  Let $M$ be a smooth domain in $[t_-, t_+] \times {S}^2$ with the metric
  $\mathrm{d} t^2 + \psi (t)^2 g_{S^2}$ where $(\log \psi)'' < 0$, $\psi (t) >
  0$ on $[t_-, t_+]$. Assume that $\partial M$ is convex with respect to the
  conformally related metric $\mathrm{d} s^2 + g_{S^2}$ where $s$ is given in
  \eqref{conformally related metric} and $g$ is another metric on $M$ which
  satisfies the comparisons of:
  \begin{enumerate}
    \item the scalar curvatures $R_g \geqslant R_{\bar{g}}$,
    
    \item the mean curvatures $H_g \geqslant H_{\bar{g}}$ of the boundary
    $\partial M$ in $M$,
    
    \item and the metrics $g \geqslant \bar{g}$,
  \end{enumerate}
  then $g = \bar{g}$.
\end{theorem}

Our approach toward items (\ref{smooth})-(\ref{item metric cone}) is by
construction of surfaces of prescribed mean curvature and prescribed
contact angles near $t = t_-$ which serves as \text{{\itshape{barriers}}}, a concept which now introduce. The existence of barriers enables us to
find a stable capillary $\mu$-bubble, if fact, $P_\pm\cap \partial M$ are natural barriers if we are in item (\ref{item simple}) of Theorem \ref{rigidity depending on structures}.

\begin{definition}
  \label{barrier condition}We say that a surface $\Sigma_+$ ($\Sigma_-$) whose
  boundary separates $\partial (P_+ \cap \partial M)$ and $\partial (P_- \cap
  \partial M)$ is an upper (lower) barrier if $H_{\Sigma_+} \geqslant \bar{h}
  |_{\Sigma_+}$ ($H_{\Sigma_-} \leqslant \bar{h} |_{\Sigma_-}$) and \
  $\gamma_{\Sigma_+} \geqslant \bar{\gamma} |_{\partial \Sigma_+ \cap \partial
  M}$ ($\gamma_{\Sigma_-} \leqslant \bar{\gamma} |_{\partial \Sigma_- \cap
  \partial M}$) along $\partial \Sigma_+$ ($\partial \Sigma_-$). We call
  $\Sigma_+$ and $\Sigma_-$ are a set of barriers if $\Sigma_+$ and $\Sigma_-$
  are respectively an upper barrier and a lower barrier and $\Sigma_+$ is
  closer to $P_+$ than $\Sigma_-$.
\end{definition}

Our construction of barriers near
$t=t_\pm$ are purely local, and in this way, our proof of items (\ref{smooth})-(\ref{item metric cone}) can be reduced to the last item of Theorem \ref{rigidity depending on structures}, that is, the case with a set of barriers.
In particular, if both $P_\pm \cap \partial M$ take the structure as specified in item (\ref{item metric cone}) of Theorem \ref{rigidity depending on structures}, we obtain a genuine generalization of Theorem \ref{llarull},
since in the case of round metric, $\psi (t) = \sin t$, $t \in [0, \pi]$ is
allowed to take zero values at $t = 0$ and $t = \pi$. 
Hu-Liu-Shi {\cite{hu-rigidity-2023}} (see also Gromov
{\cite{gromov-four-2021}}) used a $\mu$-bubble approach for Theorem
\ref{llarull}. As indicated earlier, we use the capillary version of the $\mu$-bubble approach. However, our method differs from theirs in a technical manner
when handling 
\[
\psi (t) = \sin t = t +o(|t|)
\] near $t = 0$. (Similarly, near $t=\pi$.) They constructed a family
of perturbations on the function $2 \psi' / \psi$ while our strategy is to perform a
careful tangent cone analysis near $t = 0$ or $t=\pi$. As a result of this new strategy, we are able to
generalize the Llarull Theorem \ref{llarull} to the case where the background metric $\bar{g}$
are equipped with antipodal conical points.

\begin{theorem}
  \label{conical llarull}Let $n=3$ and $(\bar{M}, \bar{g})$ be a three dimensional warped product given in
  \eqref{spherical warped product std} such that
  \[ \psi (t_{\pm}) = a_{\pm} |t - t_{\pm} | + o (|t - t_{\pm} |), \text{ } 0
     < a_{\pm} \leqslant 1, \]
  If $g$ is another smooth metric on $\bar{M}$ with possible cone singularity
  at only $t = t_{\pm}$ which satisfies the comparisons of metrics $g
  \geqslant \bar{g}$ and scalar curvatures $R_g \geqslant R_{\bar{g}}$, then
  $g = \bar{g}$.
\end{theorem}

Theorem \ref{conical llarull} directly follows from the proof of item (\ref{item metric cone}) of Theorem
\ref{rigidity depending on structures} with only slight changes and we omit its proof. See Remark
\ref{alternative to llarull}. Note that the condition $0 < a_{\pm} \leqslant 1$ 
ensures that the Ricci curvature of the tangent cone with respect to $\bar{g}$ at $t
= t_{\pm}$ is non-negative. The scalar curvature rigidity of the Llarull type for
$a_{\pm} > 1$ is an interesting question. One could also compare Theorem
\ref{conical llarull} with {\cite{chu-llarulls-2024}} where conical
singularities with respect to the metric $g$ are allowed at multiple points on
$S^n$.

Some of the essential difficulties of items (\ref{smooth}) and (\ref{item euclid cone}) are
already present in {\cite[Theorem 1.2 \text{{\itshape{(2)}}} and
\text{{\itshape{(3)}}}]{chai-scalar-2023-arxiv}}. In light of this, we only
give a sketch for their proof in Section \ref{barrier II}
, and refer relevant details to {\cite{chai-scalar-2023-arxiv}}.

It is possible that the inequalities in $(\log \psi)'' < 0$ and the convexity of $\partial_sM$
can be weakened in some cases. For instance, we can consider a convex disk $M$ in $(-\infty, \infty)\times S^2$ with the direct product metric
$\mathrm{d} t^2 + g_{S^2}$. In this case, $\log \psi$ vanishes.
The Llarull type rigidity Theorem \ref{smooth boundary llarull} is still valid for this $M$.

Now one could naturally ask what are other shapes of point singularities, in particular, asymptotics of $\psi$ near $t=t_\pm$, such that Theorem \ref{rigidity depending on structures} remain valid.
However, it is a quite intricate matter to which we do not have an answer at
the moment. It is also desirable to find a proof for higher dimensional
analogs of our results using the stable capillary $\mu$-bubbles. This seems to
be a promising direction to investigate being aware of the recent works
{\cite{cecchini-scalar-2024-arxiv,wang-scalar-mean-2024}}.

\

The article is organized as follows:

In Section \ref{sec:cap}, we introduce basics of stable capillary $\mu$-bubble
and we use it to show item (\ref{item simple}) of Theorem \ref{rigidity depending on structures}.

In Section \ref{sec:cone}, we use the tangent cone analysis at $t = t_-$ to
construct barriers and reduce item (\ref{item metric cone}) of Theorem \ref{rigidity depending on structures} to item (\ref{item simple}).

In Section \ref{barrier II}, we revisit our constructions in
{\cite{chai-scalar-2023-arxiv}} and use the techniques developed there to show items (\ref{item euclid cone}) and (\ref{smooth}) of Theorem \ref{rigidity depending on structures}.

\

\text{{\bfseries{Acknowledgments}}} X. Chai was supported by the National
Research Foundation of Korea (NRF) grant funded by the Korea government (MSIT)
(No. RS-2024-00337418) and an NRF grant No. 2022R1C1C1013511.
G. Wang was supported by the China Postdoctoral Science Foundation (No. 2024M751604).

\

\section{Stable capillary $\mu$-bubble}\label{sec:cap}

In this section, we introduce the functional \eqref{action} whose minimiser is
a stable capillary $\mu$-bubble. We introduce the \text{{\itshape{barrier}}}
condition which combining with a maximum principle ensures the existence of a
regular minimiser to \eqref{action}. By a rigidity analysis on the second
variation of \eqref{action}, we conclude the proof of item (\ref{item simple}) of Theorem \ref{rigidity depending on structures}.

\subsection{Notations}\label{general notations}We set up some notations. Let
$E \subset M$ be be a set such that $\partial E \cap M$ is a regular surface
with boundary which we name it $\Sigma$. We set
\begin{itemize}
  \item $N$, unit normal vector of $\Sigma$ pointing inside $E$;
  
  \item $\nu$, unit normal vector of $\partial \Sigma$ in $\Sigma$ pointing
  outside of $\Sigma$;
  
  \item $\eta$, unit normal vector of $\partial \Sigma$ in $\partial M$
  pointing outside of $\partial E \cap \partial M$;
  
  \item $X$: unit normal vectors of $\partial M$ in $M$ pointing outside of
  $M$;
  
  \item $\gamma$: the contact angle formed by $\Sigma$ and $\partial M$ and
  the magnitude of the angle is given by $\cos \gamma = \langle X, N \rangle$,
  
  \item $\langle Y, Z \rangle = g (Y, Z)$, the inner product of vectors $Y$
  and $Z$ with respect to the metric $g$;
  
  \item $\langle Y, Z \rangle_{\bar{g}} = \bar{g} (Y, Z)$, the inner product
  of vectors $Y$ and $Z$ with respect to the metric $\bar{g}$.
\end{itemize}
See Figure \ref{fig:notations}. We use a bar on every
quantity to denote that the quantity is computed with respect to the metric
$\bar{g}$ given in \eqref{spherical warped product}.

\begin{figure}[ht]
    \centering
	\begingroup
	\def\svgwidth{0.4\columnwidth}
	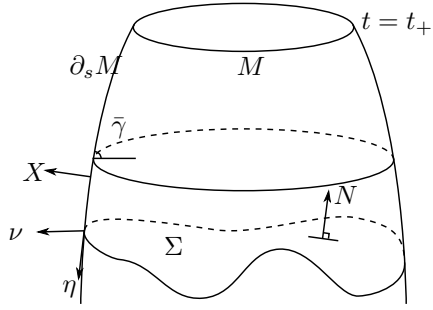
	\endgroup

    \caption{Notations.}
    \label{fig:notations}
\end{figure}

\subsection{Functional and first variation}
We fix $\bar{h} = 2 \psi'/\psi$ and $\bar{\gamma}$ to be given by \eqref{angle}.
We define the functional
\begin{equation}
  I (E) =\mathcal{H}^2 (\partial^{\ast} E \cap
  \ensuremath{\operatorname{int}}M) - \int_E \bar{h} - \int_{\partial^{\ast} E
  \cap \partial M} \cos \bar{\gamma}, \label{action}
\end{equation}
where $\partial^{\ast} E$ denotes the reduced boundary of $E$ and the
variational problem
\begin{equation}
  \mathcal{I}= \inf \{I (E) : \text{ } E \in \mathcal{E}\}, \label{variational
  problem}
\end{equation}
where $\mathcal{E}$ is the collection of contractible open subsets $E'$ such
that $P_+ \subset E'$. Let $\Sigma$ be a surface with boundary $\partial
\Sigma$ such that $\partial \Sigma$ separates $P_{\pm}$. Then $\Sigma$
separates $M$ into two components and the component closer to $P_+$ is just
$E$. We reformulate the functional \eqref{action} in terms of $\Sigma$. We
define
\begin{equation}
  F (\Sigma) = I (E) = | \Sigma | - \int_E \bar{h} - \int_{\partial E \cap
  \partial M} \cos \bar{\gamma} . \label{F action}
\end{equation}
Let $\phi_t$ be a family of immersions $\phi_t : \Sigma \to M$ such that
$\phi_t (\partial \Sigma) \subset \partial M$ and $\phi_0 (\Sigma) = \Sigma$.
Let $\Sigma_t = \phi_t (\Sigma)$ and $E_t$ be the corresponding component
separated by $\Sigma_t$. Let $Y$ be the vector field $\tfrac{\partial
\phi_t}{\partial t}$. Define $\mathcal{A} (t) = F (\Sigma_t)$ and $f = \langle
Y, N \rangle$, then by the first variation
\begin{equation}
  \mathcal{A}' (0) = \int_{\Sigma} f (H - \bar{h}) + \int_{\partial \Sigma}
  \langle Y, \nu - \eta \cos \bar{\gamma} \rangle . \label{first variation}
\end{equation}
We know that if $\Sigma$ is regular, then it is of mean curvature $\bar{h}$
and meets $\partial M$ at a prescribed angle $\bar{\gamma}$. And $E$ is called a
\text{{\itshape{capillary}}} $\mu$-\text{{\itshape{bubble}}}. The second
variation at such $\Sigma$ is
\begin{equation}
  \mathcal{A}'' (0) = Q (f, f) : = - \int_{\Sigma} (f \Delta f + (|A|^2
  +\ensuremath{\operatorname{Ric}}(N) + \partial_N \bar{h}) f^2) +
  \int_{\partial \Sigma} f (\tfrac{\partial f}{\partial \nu} - q f) .
  \label{Q}
\end{equation}
where $f \in C^{\infty} (\Sigma)$ and
\begin{equation}
  q : = \tfrac{1}{\sin \bar{\gamma}} A_{\partial M} (\eta, \eta) - \cot
  \bar{\gamma} A (\nu, \nu) + \tfrac{1}{\sin^2 \bar{\gamma}} \partial_{\eta}
  \cos \bar{\gamma} . \label{q}
\end{equation}
We define two operators
\begin{equation}
  L = - \Delta - (|A|^2 +\ensuremath{\operatorname{Ric}}(N) + \partial_N
  \bar{h}) \text{ in } \Sigma, \label{L}
\end{equation}
and
\[ B = \tfrac{\partial}{\partial \nu} - q \text{ on } \partial \Sigma .
   \label{B} \]
The surface $\Sigma$ is called \text{{\itshape{stable}}} if
\begin{equation}
  Q (f, f) \geqslant 0 \label{stability}
\end{equation}
for all $f \in C^{\infty} (M)$. The second variation \eqref{Q} is closely
related to the variation of $H - \bar{h}$ and $\cos \gamma - \cos
\bar{\gamma}$. Indeed, let $f = \langle Y, N \rangle$, we have that the first
variation of $H - \bar{h}$ is
\begin{align}
\nabla_Y (H - \bar{h}) & = L f + \nabla_{Y^{\top}} (H - \bar{h}) \\
& = - \Delta f - (|A|^2 +\ensuremath{\operatorname{Ric}}(N) + \partial_N
\bar{h}) f + \nabla_{Y^{\top}} (H - \bar{h}) . \label{first variation of
mean curvature difference}
\end{align}
And the first variation of the angle difference $\langle X, N \rangle - \cos
\bar{\gamma}$ is
\begin{align}
\nabla_Y (\cos \gamma - \cos \bar{\gamma}) = & - \sin \bar{\gamma}
\tfrac{\partial f}{\partial \nu} \\
+ (A_{\partial M} (\eta, \eta) &  - \cos \bar{\gamma} A (\nu, \nu) +
\tfrac{1}{\sin \bar{\gamma}} \partial_{\eta} \cos \bar{\gamma}) f +
\nabla_{Y^{\top}} (\langle X, N \rangle - \cos \bar{\gamma}) . \label{first
variation of angle difference}
\end{align}
For $\Sigma$, Schoen-Yau {\cite{schoen-proof-1979}} rewrote the term $|A|^2
+\ensuremath{\operatorname{Ric}} (N)$ as
\begin{equation}
  |A|^2 +\ensuremath{\operatorname{Ric}} (N) = \tfrac{1}{2} (R_g - 2 K + |A|^2
  + H^2) \label{sy rewrite}
\end{equation}
where $K$ is the Gauss curvature of $\Sigma$. Along the boundary $\partial
\Sigma$, we have the rewrite (see {\cite[Lemma 3.1]{ros-stability-1997}} or
{\cite[(4.13)]{li-polyhedron-2020}})
\begin{equation}
  \tfrac{1}{\sin \bar{\gamma}} A_{\partial M} (\eta, \eta) - \cos \bar{\gamma}
  A (\nu, \nu) = - H \cot \bar{\gamma} + \tfrac{H_{\partial M}}{\sin
  \bar{\gamma}} - \kappa \label{rs rewrite}
\end{equation}
where $\kappa$ is the geodesic curvature of $\partial \Sigma$ in $\Sigma$.

\subsection{Analysis of stability}Starting from now on, we assume that
$\Sigma$ is a regular stable capillary $\mu$-bubble in $(M, g)$ which
satisfies the assumptions of item (\ref{item simple}) of Theorem \ref{rigidity depending on structures}.

\begin{lemma}
  \label{simple consequence}Let $\Sigma$ be a regular stable capillary
  $\mu$-bubble, then $\Sigma$ is a $t$-level set.
\end{lemma}

\begin{proof}
  First, we note that the second variation $\mathcal{A}'' (0) \geqslant 0$ as
  in \eqref{Q}. First, using Schoen-Yau's rewrite \eqref{sy rewrite} we see
  that
\begin{align}
& |A|^2 +\ensuremath{\operatorname{Ric}} (N) + \partial_N \bar{h}
\\
= & \tfrac{1}{2} (R - 2 K + |A|^2 + H^2) + \partial_N \bar{h} \\
= & \tfrac{1}{2} (R - 2 K + |A^0 |^2 + \tfrac{H^2}{2} + H^2) + \partial_N
\bar{h} \\
= & \tfrac{1}{2} (R + \tfrac{3}{2} \bar{h}^2 + 2 \partial_N \bar{h}) - K +
\tfrac{1}{2} |A^0 |^2, \label{sy}
\end{align}
  where $A^0$ is the traceless part of the second fundamental form. Similarly
  using \eqref{rs rewrite}, we see
  \[ q = - H \cot \bar{\gamma} + \tfrac{H_{\partial M}}{\sin \bar{\gamma}} -
     \kappa + \tfrac{1}{\sin^2 \bar{\gamma}} \partial_{\eta} \cos \bar{\gamma}
     . \]
  We obtain by letting $f \equiv 1$ in the \eqref{stability} (also using
  \eqref{Q} and \eqref{q}),
\begin{align}
2 \pi \chi (\Sigma) & = \int_{\Sigma} K + \int_{\partial \Sigma} \kappa
\\
& \geqslant \int_{\Sigma} \left[ \tfrac{1}{2} (R + \tfrac{3}{2} \bar{h}^2
+ 2 \partial_N \bar{h}) + \tfrac{1}{2} |A^0 |^2 \right] + \int_{\partial
\Sigma} \left( \tfrac{H_{\partial M}}{\sin \bar{\gamma}} - \bar{h} \cot
\bar{\gamma} + \tfrac{1}{\sin^2 \bar{\gamma}} \partial_{\eta} \cos
\bar{\gamma} \right) \\
& \geqslant \int_{\Sigma} \tfrac{1}{2} \left( R + \tfrac{3}{2} \bar{h}^2
+ 2 \partial_N \bar{h} \right) + \int_{\partial \Sigma} \left(
\tfrac{H_{\partial M}}{\sin \bar{\gamma}} - \bar{h} \cot \bar{\gamma} +
\tfrac{1}{\sin^2 \bar{\gamma}} \partial_{\eta} \cos \bar{\gamma} \right)
\\
& \geqslant \int_{\Sigma} \tfrac{1}{2} \left( R_{\bar{g}} + \tfrac{3}{2}
\bar{h}^2 + 2 \partial_N \bar{h} \right) + \int_{\partial \Sigma} \left(
\tfrac{\bar{H}_{\partial M}}{\sin \bar{\gamma}} - \bar{h} \cot
\bar{\gamma} + \tfrac{1}{\sin^2 \bar{\gamma}} \partial_{\eta} \cos
\bar{\gamma} \right), \label{after rewrite}
\end{align}
  where in the last line we have incorporated the comparisons $R_g \geqslant
  R_{\bar{g}}$ in $M$ and $H_{\partial M} \geqslant \bar{H}_{\partial M}$ on
  $\partial M$.
  
  Now we estimate $R_{\bar{g}} + \tfrac{3}{2} \bar{h}^2 + 2 \partial_N
  \bar{h}$. We have that
  \[ \partial_N \bar{h} = \bar{g} (N, \nabla^{\bar{g}} \bar{h}) \geqslant -
     |N|_{\bar{g}} | \nabla^{\bar{g}} \bar{h} |_{\bar{g}} = |N|_{\bar{g}}
     \bar{h}', \]
  since $g \geqslant \bar{g}$, so
  \[ 1 = |N|_g \geqslant |N|_{\bar{g}}, \]
  and we get
  \[ \partial_N \bar{h} \geqslant \bar{h}' . \]
  So
  \[ R_{\bar{g}} + \tfrac{3}{2} \bar{h}^2 + 2 \partial_N \bar{h} \geqslant
     R_{\bar{g}} + \tfrac{3}{2} \bar{h}^2 + 2 \bar{h}' . \]
  For any point $x \in \Sigma$, the right hide side is just $\tfrac{2 K(p_x)}{\psi^2
    (t_x)}$. Recall that $x=(t_x,p_x)$ is the coordinate of $x \in \bar{\Sigma}_t$.
  This is by a direct calculation of the scalar curvature of the warped product metric
  \eqref{spherical warped product}.
  So
  \begin{equation}
    R_{\bar{g}} + \tfrac{3}{2} \bar{h}^2 + 2 \partial_N \bar{h} \geqslant
    \tfrac{2K(p_x)}{\psi^2 (t_x)} . \label{interior lower bound simplify}
  \end{equation}

Let $\hat{g} = \mathrm{d} s^2 + g_{{S}^2}$ and it is conformally
related to $\bar{g}$ via \eqref{conformally related metric}. Let $\hat{X}$ be
the unit outward normal of $\partial_s M$ in $M$ and $\hat{H}_{\partial_s M}$
be the mean curvature of $\partial_s M$ in $M$ with respect to $\hat{g}$.
Since $\bar{g}$ is conformal to $\hat{g}$, by a well known formula of
conformal change of mean curvature,
\begin{equation}
  \bar{H}_{\partial_s M} = \tfrac{1}{\varphi (s)} (\hat{H}_{\partial_s M} + 2
  \partial_{\hat{X}} \log \varphi) . \label{conformal of partial s M}
\end{equation}
Similarly, the mean curvature $\bar{h}$ of $\Sigma_t$ in $M$ is
\begin{equation}
  \bar{h} (t) = \tfrac{2}{\varphi (s)^2} \varphi' (s) . \label{g bar mean
  curvature of Sigma t}
\end{equation}
Hence, by \eqref{conformal of partial s M}, \eqref{g bar mean curvature of
Sigma t}, \eqref{conformally related metric} and that $\hat{g} (\partial_s,
\hat{X}) = \cos \bar{\gamma}$,
\[ \tfrac{\bar{H}_{\partial M}}{\sin \bar{\gamma}} - \bar{h} \cot \bar{\gamma}
   + \tfrac{1}{\sin^2 \bar{\gamma}} \partial_{\eta} \cos \bar{\gamma} =
   \tfrac{1}{\psi (t_x) \sin \bar{\gamma}} (\hat{H}_{\partial_s M} -
   \partial_{\psi \eta} \bar{\gamma}) . \]
Inserting \eqref{interior lower bound simplify} and the above in \eqref{after
rewrite} yields
\[ 2 \pi \chi (\Sigma) \geqslant \int_{\Sigma} \tfrac{K(p_x)}{\psi (t_x)^2}
   \mathrm{d} \sigma + \int_{\partial \Sigma} \tfrac{1}{\psi (t_x) \sin
   \bar{\gamma}} (\hat{H}_{\partial_s M} - \partial_{\psi \eta} \bar{\gamma})
   \mathrm{d} \lambda, \]
where we have written the area element $\mathrm{d} \sigma$ and line length
element $\mathrm{d} \lambda$ explicitly in the metric $g$.
The rest of the proof is deferred to the next Lemma \ref{lm:key in lower bound}.
\end{proof}

\begin{lemma}
  \label{lm:key in lower bound}Assume that $g \geqslant \bar{g}$. If $\Sigma$
  is a surface in $M$ whose boundary $\partial \Sigma$ is a simple smooth
  curve that separates $\partial (P_+ \cap \partial M)$ and $\partial (P_-
  \cap \partial M)$, then
  \begin{equation}
    \int_{\Sigma} \tfrac{K(p_x)}{\psi (t_x)^2} \mathrm{d} \sigma + \int_{\partial
    \Sigma} \tfrac{1}{\psi (t_x) \sin \bar{\gamma}} (\hat{H}_{\partial_s M} -
    \partial_{\psi \eta} \bar{\gamma}) \mathrm{d} \lambda \geqslant 2 \pi,
    \label{crucial estimate}
  \end{equation}
  where equality occurs if and only if $\Sigma$ is a $t$-level set.
\end{lemma}

\begin{proof}
It suffices to prove \eqref{crucial estimate} for
$\psi \equiv 1$ since $g \geqslant \bar{g} = \psi^2 \hat{g}$. In this case, $s=t$, we just use $t$. We also suppress
the subscript $\partial_s M$ for clarity in this proof. In addition, we assume that
every $t$-coordinate of $\Sigma$ is strictly less than $t_+$, since we can increase
$t_+$ a little.
Let $e_1$ be the unit
tangent vector of $\partial \Sigma_t$ with respect to $\hat{g}$ and $e_2$ be
the unit outward normal of $\partial \Sigma_t $ in $\partial_s M$ with
respect to $\hat{g}$. Let $T$ (resp. $\hat{T}$) be the unit tangent vector of
$\partial \Sigma$ with respect to $g$ (resp. $\hat{g}$). We recall an ingenious
inequality of {\cite[Lemma 3.2]{ko-scalar-2024}},
\begin{equation}
  \hat{H} - \nabla_{\eta} \bar{\gamma} \geqslant \langle T, e_2 \nabla_1
  \bar{\gamma} + (\hat{H} - \nabla_2 \bar{\gamma}) e_1 \rangle, \label{ko-yao}
\end{equation}
where $\langle \cdot, \cdot \rangle$ denotes the inner product with respect to
$\hat{g}$.

Let $\lambda$ be an arc-length parameter of $\partial \Sigma$ with respect to
$g$, then the length element of $\partial \Sigma$ with respect to $\hat{g}$ is
given by $|T|_{\hat{g}} \mathrm{d} \lambda =: \mathrm{d} \hat{\lambda}$ and $T
= |T|_{\hat{g}} \hat{T}$. Therefore,
\begin{align}
  & \int_{\partial \Sigma} \tfrac{1}{\sin \bar{\gamma}} (\hat{H} -
  \partial_{\eta} \bar{\gamma}) \mathrm{d} \lambda \\
  \geqslant & \int_{\partial \Sigma} \tfrac{1}{\sin \bar{\gamma}} \langle T,
  \partial_1 \bar{\gamma} e_2 + (\hat{H} - \partial_2 \bar{\gamma}) e_1
  \rangle \mathrm{d} \lambda \\
  = & \int_{\partial \Sigma} \tfrac{1}{\sin \bar{\gamma}} \langle \hat{T},
  \nabla_1 \bar{\gamma} e_2 + (\hat{H} - \nabla_2 \bar{\gamma}) e_1 \rangle
  \mathrm{d} \hat{\lambda} . 
\end{align}
Let $\hat{\eta}$ be the unit outward normal of $\partial \Sigma$ in
$\partial_s M$ with respect to $\hat{g}$. Recall the notation $\hat{A} =
\hat{A}_{\partial_s M}$, and we know that the components of $\hat{A}$ satisfy
\begin{equation}
  \partial_1 \bar{\gamma} = \hat{A}_{12}, \text{ } \partial_2 \bar{\gamma} =
  \hat{A}_{2 2}, \text{ } \hat{H} - \partial_2 \bar{\gamma} = \hat{A}_{1 1},
  \label{angles and 2ff}
\end{equation}
in the frame $\{e_1, e_2 \}$. So
\[ \langle \hat{T}, \partial_1 \bar{\gamma} e_2 + (\hat{H} - \partial_2
   \bar{\gamma}) e_1 \rangle = \langle \hat{T}, \hat{A}_{1 2} e_2 + \hat{A}_{1
   1} e_1 \rangle = \hat{A}_{1 \hat{T}} . \]
Since the orthonormal frames $\{ \hat{T}, \hat{\eta} \}$ and $\{e_1, e_2 \}$
have the same orientation, we can set $\hat{T} = a_1 e_1 + a_2 e_2$ and so
$\hat{\eta} = - a_2 e_1 + a_1 e_2$ for some $a_1$ and $a_2$ which satisfy
$a_1^2 + a_2^2 = 1$. It then follows that
\begin{equation}
  \hat{A}_{1 \hat{T}} = - (\hat{A} - \hat{H}  \hat{g}) (e_2, \hat{\eta}) = : -
  \hat{W} (e_2, \hat{\eta}), \label{newton tensor}
\end{equation}
by considering also that $\hat{H} = \hat{A}_{11} + \hat{A}_{22}$. Hence
\[ \int_{\partial \Sigma} \tfrac{1}{\sin \bar{\gamma}} (\hat{H} -
   \partial_{\eta} \bar{\gamma}) \mathrm{d} \lambda \geqslant - \int_{\partial
   \Sigma} \tfrac{1}{\sin \bar{\gamma}} \hat{W} (e_2, \hat{\eta}) \mathrm{d}
   \hat{\lambda} . \]
Suppose that $\partial \Sigma(t_+)$ (we set $\Sigma(t_{+\pm})=\Sigma_{t_{\pm}}$ to avoid double subscripts) and $\partial \Sigma$ enclose a region
$S$ in $\partial_s M$. By the divergence theorem,
\begin{align}
  & - \int_{\partial \Sigma} \tfrac{1}{\sin \bar{\gamma}} \hat{W} (e_2,
  \hat{\eta}) \mathrm{d} \hat{\lambda} \\
  = & - \int_S \hat{\nabla}^S_i \left( \hat{W}_{i j} \langle \tfrac{1}{\sin
  \bar{\gamma}} e_2, e_j \rangle \right) \mathrm{d} \hat{\sigma} -
  \int_{\partial \Sigma (t_+)} \tfrac{1}{\sin \bar{\gamma}} (\hat{A} - \hat{H}
  \hat{g}) (e_2, e_2) \mathrm{d} \hat{\lambda} \\
  = : & - \int_S \hat{\nabla}^S_i (\hat{A}_{i j} - \hat{H} \hat{g}_{i j})
  \langle \tfrac{1}{\sin \bar{\gamma}} e_2, e_j \rangle \mathrm{d}
  \hat{\sigma} - \int_S \hat{W}_{i j}  \langle \hat{\nabla}_i^S(\tfrac{1}{\sin
  \bar{\gamma}} e_2), e_j \rangle \mathrm{d} \hat{\sigma} - I_3 \\
  =: & - I_1 - I_2 - I_3, 
\end{align}
where $\hat{\nabla}^S$ is the induced connection on $S$ with respect to the
metric $\hat{g}$. For $I_1$, we use Gauss-Codazzi equation,
\[ I_1 =
  - \int_S \ensuremath{\operatorname{Ric}}_{\hat{g}} (\hat{X},
  \tfrac{1}{\sin \bar{\gamma}} e_2) \mathrm{d} \hat{\sigma}
  = - \int_S K(p_x)\cos \bar{\gamma} \mathrm{d} \hat{\sigma}
  = - \int_S K(p_x) \langle \partial_t, \hat{X} \rangle \mathrm{d} \hat{\sigma} ,
\]
where the fact that $\hat{g} = \mathrm{d}t^{2} +g_{S^{2}} $ was also used in
determining the Ricci curvature.
For $I_3$, we use \eqref{newton tensor}, we see
\[ I_3 = \int_{\partial \Sigma (t_+)} \tfrac{1}{\sin \bar{\gamma}} \hat{A}_{1
   1} \mathrm{d} \hat{\lambda} = \int_{\partial \Sigma (t_+)} \hat{\kappa}(p_x,t)
   \mathrm{d} \hat{\lambda}, \]
where $\hat{\kappa}(p,t)$ is the geodesic curvature of $\partial \Sigma_t $ in
$\Sigma_t$ at $(p,t)\in \partial \Sigma_t$.
It seems tricky to calculate $\hat{\nabla}_i \left( \tfrac{1}{\sin
\bar{\gamma}} e_2 \right)$ in $I_2$ directly at a first sight, but we can convert to terms that
are easier. By the identity,
\[ \tfrac{1}{\sin \bar{\gamma}} e_2 = \partial_t - \tfrac{\cos
   \bar{\gamma}}{\sin \bar{\gamma}} \hat{\nu} = \partial_t - \tfrac{\cos
   \bar{\gamma}}{\sin^2 \bar{\gamma}} \hat{X} + \tfrac{\cos^2
   \bar{\gamma}}{\sin^2 \bar{\gamma}} \partial_t  \]
at $x\in \partial_sM$, we see
\begin{align}
  &  \hat{W}_{i j} \langle \hat{\nabla}_i (\tfrac{1}{\sin \bar{\gamma}}
  e_2), e_j \rangle \\
  =&  \hat{W}_{i j} \left\langle \hat{\nabla}_i \left( \partial_t -
  \tfrac{\cos \bar{\gamma}}{\sin^2 \bar{\gamma}} \hat{X} + \tfrac{\cos^2
  \bar{\gamma}}{\sin^2 \bar{\gamma}} \partial_t \right), e_j \right\rangle
  \\
  =&  - \tfrac{\cos \bar{\gamma}}{\sin^2 \bar{\gamma}} \hat{W}_{i j}
  \hat{A}_{i j} + \hat{W}_{i j} \partial_i \left( \tfrac{\cos^2
  \bar{\gamma}}{\sin^2 \bar{\gamma}} \right) \langle \partial_t, e_j \rangle
  \\
  =&  - \tfrac{\cos \bar{\gamma}}{\sin^2 \bar{\gamma}} \hat{W}_{i j}
  \hat{A}_{i j} - 2 \hat{W}_{i 2} \tfrac{\cos \bar{\gamma}}{\sin^3
  \bar{\gamma}} \partial_i \bar{\gamma} \langle \partial_t, e_2 \rangle
\end{align}
at $x$. It is now a tedious task to check from the definition of $\hat{W}$,
$\langle \partial_t, e_2 \rangle = - \sin \bar{\gamma}$ and \eqref{angles and
2ff} that the above vanishes. Therefore, $I_2 = 0$ and to sum up, we have
shown that
\begin{equation}
  \int_{\partial \Sigma} \tfrac{1}{\sin \bar{\gamma}} (\hat{H} -
  \partial_{\eta} \bar{\gamma}) \mathrm{d} \lambda \geqslant - \int_S K(p_x) \langle
  \partial_t, \hat{X} \rangle \mathrm{d} \hat{\sigma} + \int_{\partial \Sigma
  (t_+)} \hat{\kappa}(p_x,t) \mathrm{d} \hat{\lambda} . \label{bdry term estimate}
\end{equation}
Now we set the region enclosed by $\Sigma (t_+)$, $\Sigma$ and $\partial_s M$
to be $\Omega$. Let $ \hat G = \operatorname{Ric}_{\hat g} - \frac{1}{2} R_{\hat g} \hat g $. Using the divergence free property of $\hat G$ (twice-contracted Gauss-Codazzi equation), $\partial_t$, and the divergence theorem, 
\[ 0 = \int_{\Omega} \hat{\nabla}_i (\hat{G}_{i j} (\partial_t)_j) =
  \int_{\partial \Omega} \hat{G}( \partial_t, \hat{X}) = - \frac{1}{2} \int_{\partial \Omega} R_{\hat g} \langle \partial_t, \hat X \rangle ,
= - \int_{\partial \Omega} K(p_x) \langle \partial_t, \hat X \rangle, \]
where $\hat{X}$ now also denotes the unit outward normal of $\partial \Omega$
with respect to $\hat{g}$ and $(\partial_t)_j$ denotes the $j$-th component of the vector field $\partial_t$. Note that $\partial \Omega = S \cup \Sigma (t_+)
\cup \Sigma$ and so
\[ 0 = \int_{\Sigma (t_+)} K(p_x) \langle \partial_t, \partial_t \rangle -
   \int_{\Sigma} K(p_x) \langle \partial_t, \hat{N} \rangle + \int_S K(p_x) \langle
   \partial_t, \hat{X} \rangle, \]
and it follows from that $K(p_x)>0$ that
\begin{equation}
   \int_{\Sigma }K(p_x) \geqslant  \int_{\Sigma} K(p_x) \langle \partial_t, \hat{N}
  \rangle = \int_{ \Sigma (t_+) } K(p_x) + \int_S K(p_x) \langle \partial_t, \hat{X}
  \rangle . \label{interior estimate}
\end{equation}
Finally, it follows from \eqref{bdry term estimate} and \eqref{interior
estimate} that
\[ \int_{ \Sigma } K(p_x) + \int_{\partial \Sigma} \tfrac{1}{\sin \bar{\gamma}}
  (\hat{H} - \partial_{\eta} \bar{\gamma}) \mathrm{d} \lambda \geqslant
  \int_{ \Sigma (t_+) } K(p_x) + \int_{\partial \Sigma (t_+)} \hat{\kappa}(p_x,t)
   \mathrm{d} \hat{\lambda} . \]
An application of the Gauss-Bonnet theorem 
on the right hand finishes the proof of \eqref{crucial
  estimate}. The equality case is easy to trace.
\end{proof}

\begin{remark}
    The convexity of $\partial_s M$ in  $(M,\mathrm{d}s^2 + g_{S_2})$ with $\mathrm{d}s^2 + g_{S_2}$ given in \eqref{conformally related metric} is used in the inequality \eqref{ko-yao}.
\end{remark}

\subsection{Infinitesimally rigid surface}

The surface $\Sigma$ is a stable capillary $\mu$-bubble has more consequences
than the mere Lemma \ref{simple consequence}. We can conclude that $\Sigma$ is
a so-called infinitesimally rigid surface. See Definition \ref{infinitesimal
rigid surface}.

All inequalities are in fact equalities by Lemma \ref{lm:key in lower bound}
and tracing the equalities in \eqref{after rewrite}, we arrive that
\begin{equation}
  R_g = R_{\bar{g}}, N = \bar{N}, |A^0 | = 0 \text{ in } \Sigma
  \label{infinitesimal interior}
\end{equation}
and
\begin{equation}
  H_{\partial M} = \bar{H}_{\partial M} \text{ along }
  \partial \Sigma . \label{infinitesimal boundary}
\end{equation}
It then follows from the equality case of Lemma \ref{lm:key in lower bound} that
\begin{equation}
 t_x = t_0 \text{ at all } x \in \bar{\Sigma} \label{eq:metric
  and level}
\end{equation}
for some constant $t_0 \in [t_-,t_+]$. Because $\Sigma$ is stable
(equivalently $Q (f, f) \geqslant 0$), so the eigenvalue problem
\begin{equation}
  \left\{\begin{array}{ll}
    L f & = \mu f \text{ in } \Sigma\\
    B f & = 0 \text{ on } \partial \Sigma
  \end{array}\right. \label{eigen problem}
\end{equation}
has a non-negative first eigenvalue $\mu_1 \geqslant 0$.

The analysis now is similar to {\cite{fischer-colbrie-structure-1980}}.
Letting $f \equiv 1$ in \eqref{stability}, using \eqref{infinitesimal
interior}, \eqref{infinitesimal boundary} and \eqref{eq:metric and level}, we
get
\begin{align}
Q (1, 1) = & \int_{\Sigma} \left[ K - \tfrac{1}{2} (R + \tfrac{3}{2}
\bar{h}^2 + 2 \partial_N \bar{h}) \right] \\
& \quad + \int_{\partial \Sigma} \left[ \kappa - (\tfrac{H_{\partial
M}}{\sin \bar{\gamma}} - \bar{h} \cot \bar{\gamma} - \tfrac{1}{\sin
\bar{\gamma}} \tfrac{\partial \bar{\gamma}}{\partial \eta}) \right] = 0.
\label{constant in bilinear form}
\end{align}
And so the first eigenvalue $\mu_1$ is zero, and the constant 1 is its
corresponding eigenfunction.

By \eqref{infinitesimal interior} and \eqref{sy}, the stability operator $L$
reduces to
\[ L = - \Delta - \left( \tfrac{K(p_x)}{\psi (t_0)^2} - K \right) ; \]
by considering \eqref{infinitesimal boundary} and that $t_x =t_0$
, the boundary stability operator $B$ reduces to
\[ B = \partial_{\nu} - \left( \frac{\hat{\kappa}(p_x,t_0)}{\psi
   (t_0)} - \kappa \right) . \]
Putting $f = 1$ and $\mu_1 = 0$ in the eigenvalue problem \eqref{eigen
problem}, we get
\begin{equation}
  K = \tfrac{K(p_x)}{\psi^2 (t_0)} \text{ in } \Sigma, \text{ } \kappa =
  \tfrac{\hat{\kappa}(p_x,t_0)}{ \psi (t_0)} \text{ on } \partial
  \Sigma . \label{vanishing gauss and constant geodesic}
\end{equation}
Now we
summarize the properties of $\Sigma$ in the definition of an
\text{{\itshape{infinitesimally}}} \text{{\itshape{rigid}}}
\text{{\itshape{surface}}}.

\begin{definition}
  \label{infinitesimal rigid surface}We say that $\Sigma$ is
  \text{{\itshape{infinitesimally rigid}}} if it satisfies
  \eqref{infinitesimal interior}, \eqref{infinitesimal boundary},
  \eqref{eq:metric and level} and \eqref{vanishing gauss and constant
  geodesic}.
\end{definition}

\subsection{Capillary foliation of constant $H - \bar{h}$}\label{construction
of CMC foliation}

See for instance the works {\cite{ye-foliation-1991}},
{\cite{bray-rigidity-2010}} and {\cite{ambrozio-rigidity-2015}} on
constructing CMC foliations. First, we construct a foliation with prescribed
angles $\bar{\gamma}$ whose leaf is of constant $H - \bar{h}$. Let $\phi (x,
t)$ be a local flow of a vector field $Y$ which is tangent to $\partial M$ and
transverse to $\Sigma$ and that $\langle Y, N \rangle = 1$.

In the following theorem, we only require that the scalar curvature of $(M,
g)$ and the mean curvature of $\partial M$ are bounded below.

\begin{theorem}
  \label{foliation near infinitesimal}Suppose $(M, g)$ is a three manifold
  with boundary, if $\Sigma$ is an infinitesimally rigid surface, then there
  exists $\varepsilon > 0$ and a function $w (x, t)$ on $\Sigma \times (-
  \varepsilon, \varepsilon)$ such that for each $t \in (- \varepsilon,
  \varepsilon)$, the surface
  \begin{equation}
    \Sigma_t = \{\phi (x, w (x, t)) : x \in \Sigma\}
  \end{equation}
  is a surface of constant $H - \bar{h}$ intersecting $\partial M$ with
  prescribed angle $\bar{\gamma}$. Moreover, for every $x \in \Sigma$ and
  every $t \in (- \varepsilon, \varepsilon)$,
  \begin{equation}
    w (x, 0) = 0, \text{ } \int_{\Sigma} (w (x, t) - t) = 0 \text{ and }
    \tfrac{\partial}{\partial t} w (x, t) |_{t = 0} = 1.
  \end{equation}
\end{theorem}

\begin{proof}
  Given a function in the H{\"o}lder space $C^{2, \alpha} (\Sigma) \cap C^{1,
  \alpha} (\bar{\Sigma})$ ($0 < \alpha < 1$), we consider
  \[ \Sigma_u = \{\phi (x, u (x)) : x \in \Sigma\}, \]
  which is a properly embedded surface if the norm of $u$ is small enough. We
  use the subscript $u$ to denote the quantities associated with $\Sigma_u$.
  
  Consider the space
  \[ \mathcal{Y}= \left\{ u \in C^{2, \alpha} (\Sigma) \cap C^{1, \alpha}
     (\bar{\Sigma}) : \int_{\Sigma} u = 0 \right\} \]
  and
  \[ \mathcal{Z}= \left\{ u \in C^{0, \alpha} (\Sigma) : \int_{\Sigma} u = 0
     \right\} . \]
  Given small $\delta > 0$ and $\varepsilon > 0$, we define the map
  \[ \Phi : (- \varepsilon, \varepsilon) \times B (0, \delta) \to \mathcal{Z}
     \times C^{0, \alpha} (\partial \Sigma) \]
  given by
\begin{align}
& \Phi (t, u) \\
= & \left( (H_{t + u} - \bar{h}_{t + u}) - \tfrac{1}{| \Sigma |}
\int_{\Sigma} (H_{t + u} - \bar{h}_{t + u}), \langle X_{t + u}, N_{t + u}
\rangle - \cos \bar{\gamma}_{t + u} \right) . \label{phi}
\end{align}
Here, $B(0,\delta)$ is a ball of radius $\delta>0$ centered at the zero function in $\mathcal{Y}$.
  For each $v \in \Sigma$, the map
  \[ f : (x, s) \in \Sigma \times (- \varepsilon, \varepsilon) \to \phi (x, s
     v (x)) \in M \]
  gives a variation with
  \[ \tfrac{\partial f}{\partial s} |_{s = 0} = \tfrac{\partial}{\partial s}
     \phi (x, s v (x)) |_{s = 0} = v N. \]
  Since $\Sigma$ is infinitesimally rigid and using also \eqref{first
  variation of mean curvature difference} and \eqref{first variation of angle
  difference}, we obtain that
  \[ D \Phi_{(0, 0)} (0, v) = \tfrac{\mathrm{d}}{\mathrm{d} s} \Phi (0, s v)
     |_{s = 0} = \left( - \Delta v + \tfrac{1}{| \Sigma |} \int_{\partial
     \Sigma} \Delta v, - \sin \bar{\gamma} \tfrac{\partial v}{\partial \nu}
     \right) . \]
  It follows from the elliptic theory for the Laplace operator with Neumann
  type boundary conditions that $D \Phi (0, 0)$ is an isomorphism when
  restricted to $0 \times \mathcal{Y}$.
  
  Now we apply the implicit function theorem: For some smaller $\varepsilon$,
  there exists a function $u (t) \in B (0, \delta) \subset \mathcal{X}$, $t
  \in (- \varepsilon, \varepsilon)$ such that $u (0) = 0$ and $\Phi (t, u (t))
  = \Phi (0, 0) = (0, 0)$ for every $t$. In other words, the surfaces
  \[ \Sigma_{t + u (t)} = \{\phi (x, t + u (t)) : x \in \Sigma\} \]
  are of constant $H - \bar{h}$ with prescribed angles $\bar{\gamma}$.
  
  Let $w (x, t) = t + u (t) (x)$ where $(x, t) \in \Sigma \times (-
  \varepsilon, \varepsilon)$. By definition, $w (x, 0) = 0$ for every $x \in
  \Sigma$ and $w (\cdot, t) - t = u (t) \in B (0, \delta) \subset \mathcal{X}$
  for every $t \in (- \varepsilon, \varepsilon)$. Observe that the map $s
  \mapsto \phi (x, w (x, s))$ gives a variation of $\Sigma$ with variational
  vector field is given by
  \[ \tfrac{\partial \phi}{\partial t} \tfrac{\partial w}{\partial s} |_{s =
     0} = \tfrac{\partial w}{\partial s} |_{s = 0} Y. \]
  Since for every $t$ we have that
\begin{align}
0 = & \Phi (t, u (t)) \\
= & \left( (H_{w (\cdot, t)} - \bar{h}_{w (\cdot, t)}) - \tfrac{1}{|
\Sigma |} \int_{\Sigma} (H_{w (\cdot, t)} - \bar{h}_{w (\cdot, t)}),
\langle X_{t + u}, N_{t + u} \rangle - \cos \bar{\gamma}_{t + u} \right),
\end{align}
  by taking the derivative at $t = 0$ we conclude that
  \[ \langle \tfrac{\partial w}{\partial t} |_{t = 0} Y, N \rangle =
     \tfrac{\partial w}{\partial t} |_{t = 0} \]
  satisfies the homogeneous Neumann problem. Therefore, it is constant on
  $\Sigma$. Since
  \[ \int_{\Sigma} (w (x, t) - t) = \int_{\Sigma} u (x, t) = 0 \]
  for every $t$, by taking derivatives at $t = 0$ again, we conclude that
  \[ \int_{\Sigma} \tfrac{\partial w}{\partial t} |_{t = 0} = | \Sigma | . \]
  Hence, $\tfrac{\partial w}{\partial t} |_{t = 0} = 1$. Taking $\varepsilon$
  small, we see that $\phi (x, w (x, t))$ parameterize a foliation near
  $\Sigma$.
\end{proof}

\begin{theorem}
  \label{H ode}There exists a continuous function $\Psi (t)$ such that
  \[ \tfrac{\mathrm{d}}{\mathrm{d} t} \left( \exp (- \int_0^t \Psi (\tau)
     \mathrm{d} \tau) (H - \bar{h}) \right) \leqslant 0. \]
\end{theorem}

\begin{proof}
  Let $\psi : \Sigma \times I \to M$ parameterize the foliation, $Y =
  \tfrac{\partial \psi}{\partial t}$, $v_t = \langle Y, N_t \rangle$. Then
  \begin{equation}
    - \tfrac{\mathrm{d}}{\mathrm{d} t} (H - \bar{h}) = \Delta_t v_t +
    (\ensuremath{\operatorname{Ric}}(N_t) + |A_t |^2) v_t + v_t \nabla_{N_t}
    \bar{h} \text{ in } \Sigma_t, \label{s variation}
  \end{equation}
  and
  \begin{equation}
    \tfrac{\partial v_t}{\partial \nu_t} = [- \cot \bar{\gamma}  A_t (\nu_t,
    \nu_t) + \tfrac{1}{\sin \bar{\gamma}} A_{\partial M} (\eta_t, \eta_t) +
    \tfrac{1}{\sin^2 \bar{\gamma}} \nabla_{\eta_t} \cos \bar{\gamma}] v_t .
    \label{boundary derivative v t for foliation}
  \end{equation}
  By shrinking the interval if needed, we assume that $v_t > 0$ for $t \in
  I$. By multiplying of \eqref{s variation} and integrate on $\Sigma_t$, we
  deduce by integration by parts and applying the Schoen-Yau rewrite \eqref{sy
  rewrite} that
\begin{align}
& - (H - \bar{h})' \int_{\Sigma_t} \tfrac{1}{v_t} \\
= & \int_{\Sigma_t} \tfrac{\Delta_t v_t}{v_t} +
(\ensuremath{\operatorname{Ric}}(N_t) + |A_t |^2 + \nabla_{N_t} \bar{h})
\\
= & \int_{\partial \Sigma_t} \tfrac{1}{v_t} \tfrac{\partial v_t}{\partial
\nu_t} + \tfrac{1}{2} \int_{\Sigma_t} (R_g + |A_t |^2 + H_t^2 + 2
\nabla_{N_t} \bar{h}) - \int_{\Sigma_t} K_{\Sigma_t} + \int_{\Sigma_t}
\tfrac{| \nabla v_t |^2}{v_t^2} .
\end{align}
  Let $\chi = A - \tfrac{1}{2} \bar{h} \sigma$, we have that
\begin{align}
& |A_t |^2 \\
= & | \chi + \tfrac{1}{2} \bar{h} \sigma |^2 \\
= & | \chi |^2 + \langle \chi, \bar{h} \sigma \rangle + \tfrac{1}{2}
\bar{h}^2, \\
= & | \chi^0 |^2 + \tfrac{1}{2} (\ensuremath{\operatorname{tr}}_{\sigma}
\chi)^2 + \bar{h} \ensuremath{\operatorname{tr}}_{\sigma} \chi +
\tfrac{1}{2} \bar{h}^2,
\end{align}
  where $\chi^0$ is the traceless part of $\chi$. Also,
  \[ H^2 = (\ensuremath{\operatorname{tr}}_{\sigma} \chi + \bar{h})^2 =
     (\ensuremath{\operatorname{tr}}_{\sigma} \chi)^2 +
     2\ensuremath{\operatorname{tr}}_{\sigma} \chi \bar{h} + \bar{h}^2 . \]
  So
\begin{align}
& - (H - \bar{h})' \int_{\Sigma_t} \tfrac{1}{v_t} \\
= & \int_{\partial \Sigma_t} \tfrac{1}{v_t} \tfrac{\partial v_t}{\partial
\nu_t} + \tfrac{1}{2} \int_{\Sigma_t} (R_g + |A_t |^2 + H_t^2 + 2
\nabla_{N_t} \bar{h}) - \int_{\Sigma_t} K_{\Sigma_t} + \int_{\Sigma_t}
\tfrac{| \nabla v_t |^2}{v_t^2} \\
= & \int_{\partial \Sigma_t} \tfrac{1}{v_t} \tfrac{\partial v_t}{\partial
\nu_t} + \tfrac{1}{2} \int_{\Sigma_t} (R_g + \tfrac{3}{2} \bar{h}^2 + 2
\nabla_{N_t} \bar{h}) \\
& + \tfrac{1}{2} \int_{\Sigma_t} | \chi^0 |^2 + \tfrac{3}{2}
(\ensuremath{\operatorname{tr}}_{\sigma} \chi)^2 + 3 \bar{h}
\ensuremath{\operatorname{tr}}_{\sigma} \chi - \int_{\Sigma_t}
K_{\Sigma_t} + \int_{\Sigma_t} \tfrac{| \nabla v_t |^2}{v_t^2} \\
\geqslant & \int_{\partial \Sigma_t} \tfrac{1}{v_t} \tfrac{\partial
v_t}{\partial \nu_t} + \int_{\Sigma_t} \tfrac{K(p_x)}{\psi^2 (t_x)} +
\tfrac{3}{2} (H - \bar{h}) \int_{\Sigma_t}  \bar{h} - \int_{\Sigma_t}
K_{\Sigma_t},
\end{align}
  where in the last line we have also used the bound \eqref{interior lower
  bound simplify}. Now we use \eqref{boundary derivative v t for foliation}
  and also the rewrite \eqref{rs rewrite}, we see that
\begin{align}
& - (H - \bar{h})' \int_{\Sigma_t} \tfrac{1}{v_t} \\
\geqslant & \int_{\partial \Sigma_t} [- \cot \bar{\gamma}  A_t (\nu_t,
\nu_t) + \tfrac{1}{\sin \bar{\gamma}} A_{\partial M} (\eta_t, \eta_t) +
\tfrac{1}{\sin^2 \bar{\gamma}} \nabla_{\eta_t} \cos \bar{\gamma}]
\\
& \quad + \int_{\Sigma_t} \tfrac{K(p_x)}{\psi^2 (t_x)} + \tfrac{3}{2} (H -
\bar{h}) \int_{\Sigma_t}  \bar{h} - \int_{\Sigma_t} K_{\Sigma_t}
\\
\geqslant & \int_{\partial \Sigma_t} [- \kappa_{\partial \Sigma_t} - H (t)
\cot \bar{\gamma} + \tfrac{1}{\sin \bar{\gamma}} H_{\partial M} +
\tfrac{1}{\sin^2 \bar{\gamma}} \nabla_{\eta_t} \cos \bar{\gamma}]
\\
& \quad + \int_{\Sigma_t} \tfrac{K(p_x)}{\psi^2 (t_x)} + \tfrac{3}{2} (H -
\bar{h}) \int_{\Sigma_t}  \bar{h} - \int_{\Sigma_t} K_{\Sigma_t} \\
= & - \left( \int_{\Sigma_t} K_{\Sigma_t} + \int_{\partial \Sigma_t}
\kappa_{\partial \Sigma_t} \right) + \left[ \int_{\Sigma_t}
\tfrac{K(p_x)}{\psi^2 (t_x)} + \int_{\partial \Sigma_t} \left( \tfrac{1}{\sin
\bar{\gamma}} H_{\partial M} - \bar{h} \cot \bar{\gamma} +
\tfrac{1}{\sin^2 \bar{\gamma}} \nabla_{\eta_t} \cos \bar{\gamma} \right)
\right] \\
& \quad + \tfrac{3}{2} (H - \bar{h}) \int_{\Sigma_t}  \bar{h} - (H -
\bar{h}) \int_{\partial \Sigma_t} \cot \bar{\gamma} .
\end{align}
  It follows from Lemma \ref{lm:key in lower bound} and the proof of Lemma \ref{simple consequence} that the second term in the big bracket is
  bounded below by $2 \pi$. Using also the Gauss-Bonnet theorem on the first term in the bracket, we see that
  \begin{equation}
    - (H - \bar{h})' \int_{\Sigma_t} \tfrac{1}{v_t} \geqslant (H - \bar{h})
    (\tfrac{3}{2} \int_{\Sigma_t} \bar{h} - \int_{\partial \Sigma_t} \cot
    \bar{\gamma}) .
  \end{equation}
  Let
  \begin{equation}
    \Psi (t) = \left( \int_{\Sigma_t} \tfrac{1}{v_t} \right)^{- 1}
    (\int_{\partial \Sigma_t} \cot \bar{\gamma} - \tfrac{3}{2} \int_{\Sigma_t}
    \bar{h}), \label{Psi}
  \end{equation}
  then note that we have assume that $v_t > 0$ near $t = 0$, so $H - \bar{h}$
  satisfies the ordinary differential inequality
  \begin{equation}
    (H - \bar{h})' - \Psi (t) (H - \bar{h}) \leqslant 0. \label{original ode}
  \end{equation}
  We see then
  \[ \tfrac{\mathrm{d}}{\mathrm{d} t} \left( \exp \left( - \int_0^t \Psi
     (\tau) \mathrm{d} \tau \right) (H - \bar{h}) \right) \leqslant 0. \]
  So the function $\exp (- \int_0^t \Psi (\tau) \mathrm{d} \tau) (H -
  \bar{h})$ is non-increasing.
\end{proof}

\subsection{From local foliation to rigidity}

Let $\Sigma_t$ be the constant mean curvature surfaces with prescribed contact
angles $\bar{\gamma}$ with $\partial M$.

\begin{proposition}
  \label{leaf infi rigid}Every $\Sigma_t$ constructed in Theorem
  \ref{foliation near infinitesimal} is infinitesimally rigid.
\end{proposition}

\begin{proof}
  Let $\Omega_t$ be the component of $M\backslash \Sigma_t$ whose closure
  contains $P_+ \cap \partial M$. We abuse the notation and define
  \[ F (t) = | \Sigma_t | - \int_{\Omega_t} \bar{h} - \int_{\partial \Omega_t}
     \cos \bar{\gamma} . \]
  By the first variation formula \eqref{first variation},
  \[ F (t_2) - F (t_1) = \int_{t_1}^{t_2} \mathrm{d} t \int_{\Sigma_t} (H -
     \bar{h}) v_t . \]
  By Theorem \ref{H ode},
  \[ H - \bar{h} \leqslant 0 \text{ if } t \geqslant 0 ; \text{ } H - \bar{h}
     \geqslant 0 \text{ if } t \leqslant 0, \]
  which in turn implies that
  \[ F (t) \leqslant 0 \text{ if } t \geqslant 0 ; \text{ } F (t) \leqslant 0
     \text{ if } t \leqslant 0. \]
  However, $\Omega_t$ is a minimiser to the functional \eqref{action}, hence
  \[ F (t) \equiv F (0) . \]
  It then follows every $\Sigma_t$ is a minimiser, hence infinitesimally
  rigid.
\end{proof}

We can conclude the proof of item (\ref{item simple}) of Theorem \ref{rigidity depending on structures}.

\begin{proof}[Proof of item (\ref{item simple}) of Theorem \ref{rigidity depending on structures}]
  We note easily by the assumptions of item (\ref{item simple}) of Theorem \ref{rigidity depending on structures} that
  $\Sigma_{\pm} = P_{\pm} \cap \partial M$ are a set of barriers (see Definition
  \ref{barrier condition}), by the maximum principle, there exists a minimiser
  $E$ to \eqref{variational problem} such that $E$ is either empty or
  $\partial E\backslash \partial_s M$ or lies entirely away from $P_{\pm}$.
  Without loss of generality, we assume that $\Sigma = \partial E \cap
  \ensuremath{\operatorname{int}}M$ non-empty. By \cite{de-philippis-regularity-2015}, $\Sigma$ is a regular stable surface of prescribed mean curvature $\bar h$ and prescribed contact angle $\bar\gamma$. Moreover, the second variation $\mathcal{A}'' (0) \geqslant 0$ in
  \eqref{Q} for any smooth family $\Sigma_s$ such that $\Sigma_0 = \Sigma$.
  
  Let $Y = \tfrac{\mathrm{d}}{\mathrm{d} t} \phi (x, w (x, t))$ where $\phi$
  and $w$ are as Theorem \ref{foliation near infinitesimal}, we show first
  that $N_t$ is conformal. It suffices to show that $Y^{\bot}$ is conformal.
  
  Since every $\Sigma_t$ is infinitesimally rigid by Proposition \ref{leaf
  infi rigid}, from \eqref{eigen problem} and \eqref{constant in bilinear
  form}, we know that $\langle Y, N_t \rangle$ is a constant. Let
  $\partial_i$, $i = 1, 2$ be vector fields induced by local coordinates on
  $\Sigma$, $\partial_i$ also extends to a neighborhood of $\Sigma$ via the
  diffeomorphism $\phi$. We have $\nabla_{\partial_i} \langle Y, N \rangle =
  0$. Note that $\Sigma_t$ are umbilical with constant mean curvature
  $\bar{h}$, so
  \[ \nabla_{\partial_i} N \equiv \tfrac{1}{2} \bar{h} \partial_i \]
  and
\begin{align}
0 & = \nabla_{\partial_i} \langle Y, N \rangle \\
& = \langle \nabla_{\partial_i} Y, N \rangle + \langle Y,
\nabla_{\partial_i} N \rangle \\
& = \langle \nabla_{\partial_i} Y, N \rangle + \tfrac{1}{2} \bar{h}
\langle Y, \partial_i \rangle .
\end{align}
  On the other hand,
\begin{align}
0 & = \langle \nabla_{\partial_i} Y, N \rangle = \langle \nabla_Y
\partial_i, N \rangle \\
& = Y \langle \partial_i, N \rangle - \langle \partial_i, \nabla_Y N
\rangle \\
& = - \langle \partial_i, \nabla_Y N \rangle \\
& = - \langle \partial_i, \nabla_{Y^{\top}} N \rangle - \langle
\partial_i, \nabla_{Y^{\bot}} N \rangle \\
& = - \tfrac{1}{2} \bar{h} \langle Y^{\top}, \partial_i \rangle - \langle
\partial_i, \nabla_{Y^{\bot}} N \rangle .
\end{align}
  Combining the two equations above, we conclude that $\nabla_{Y^{\bot}} N =
  0$ which implies that $\Sigma$ foliates a warped product under the
  diffeomorphism $\phi$ (parameterized by $t$). Considering that the induced
  metric on $\Sigma$ agrees with the induced metric from $\bar{g}$, we
  conclude that $g = \bar{g}$.
\end{proof}

\section{Construction of barriers (I)}\label{sec:cone}

In this section, we prove item (\ref{item metric cone}) of Theorem \ref{rigidity depending on structures}. Our strategy is to construct a
surface $\Sigma_-$ ($\Sigma_+$) which 
serves as a lower (upper) barrier, and to use item (\ref{item simple}) of Theorem \ref{rigidity depending on structures} to finish the proof. This
section is occupied by such a construction of $\Sigma_-$.

\subsection{Setting up coordinates and notations}

For convenience, we set $t_- = 0$. As before, for any $t>0$, we set $\Sigma_t$ to be the $t$-level set of $t$ and $\Omega_t$ to be the $t$-sublevel set, that is, all points of $M$ which lie below $\Sigma_t$. Since both $(M, g)$ and $(M, \bar{g})$ has
cone structures near where $t_- = 0$ where each cross-section of the cone is
a topological disk and it collapses to a point which we denote by $p_0$. 

In the following subsections, we construct graphical perturbations
$\Sigma_{t, t^2 u}$ of $\Sigma_t$. Let $\Sigma_{t , t^2 u}$ be the surface
which consists of points $x + t^2 u (x, t) N_t (x)$ where $N_t$ is the unit
normal of $\Sigma_t$ with respect to the metric $g$ at $x \in \Sigma_t$. The
boundary $\partial \Sigma_{t , t^2 u}$ might not lie in $\partial_s M$, we can
compensate this by expanding or shrinking $\Sigma_{t , t^2 u}$ a little, and we still
denote the resulting surface $\Sigma_{t, t^2 u}$.

We use a $t$ subscript on every geometric quantity on $\Sigma_t$ and a $t, t^2
u$ subscript on every geometric quantity on $\Sigma_{t, t^2 u}$. We will
explicitly indicate when there was confusion or change.
\begin{figure}[ht]
    \centering
	\begingroup
	\def\svgwidth{0.8\columnwidth}
\begingroup%
  \makeatletter%
  \providecommand\color[2][]{%
    \errmessage{(Inkscape) Color is used for the text in Inkscape, but the package 'color.sty' is not loaded}%
    \renewcommand\color[2][]{}%
  }%
  \providecommand\transparent[1]{%
    \errmessage{(Inkscape) Transparency is used (non-zero) for the text in Inkscape, but the package 'transparent.sty' is not loaded}%
    \renewcommand\transparent[1]{}%
  }%
  \providecommand\rotatebox[2]{#2}%
  \newcommand*\fsize{\dimexpr\f@size pt\relax}%
  \newcommand*\lineheight[1]{\fontsize{\fsize}{#1\fsize}\selectfont}%
  \ifx\svgwidth\undefined%
    \setlength{\unitlength}{680.31496063bp}%
    \ifx\svgscale\undefined%
      \relax%
    \else%
      \setlength{\unitlength}{\unitlength * \real{\svgscale}}%
    \fi%
  \else%
    \setlength{\unitlength}{\svgwidth}%
  \fi%
  \global\let\svgwidth\undefined%
  \global\let\svgscale\undefined%
  \makeatother%
  \begin{picture}(1,0.5)%
    \lineheight{1}%
    \setlength\tabcolsep{0pt}%
    \put(0,0){\includegraphics[width=\unitlength,page=1]{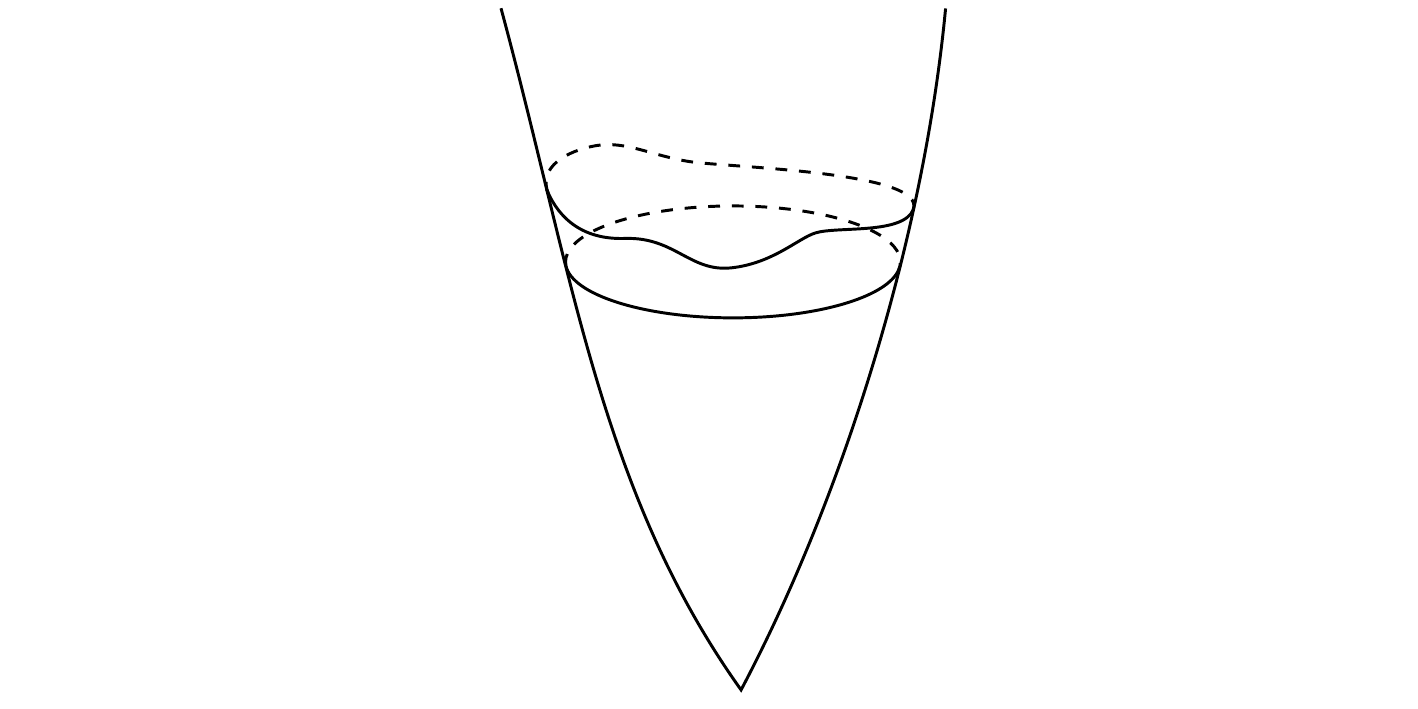}}%
    \put(0.51371514,-0.0235302){\color[rgb]{0,0,0}\makebox(0,0)[lt]{\lineheight{0}\smash{\begin{tabular}[t]{l}$P_0$\end{tabular}}}}%
    \put(0.42137305,0.00872904){\color[rgb]{0,0,0}\makebox(0,0)[lt]{\lineheight{0}\smash{\begin{tabular}[t]{l}$t=0$\end{tabular}}}}%
    \put(0.32216695,0.29923846){\color[rgb]{0,0,0}\makebox(0,0)[lt]{\lineheight{0}\smash{\begin{tabular}[t]{l}$\Sigma_t$\end{tabular}}}}%
    \put(0.29631618,0.36021392){\color[rgb]{0,0,0}\makebox(0,0)[lt]{\lineheight{0}\smash{\begin{tabular}[t]{l}$\Sigma_{t,t^2u}$\end{tabular}}}}%
    \put(0,0){\includegraphics[width=\unitlength,page=2]{conical-singularity.pdf}}%
  \end{picture}%
\endgroup%

	\endgroup

    \caption{Construction of $\Sigma_{t,t^2u}$.}
    \label{fig:conical-singularity}
\end{figure}

\subsection{Capillary foliation with constant $H -
\bar{h}$}\label{foliation}We assume that $(M, g)$ and $(M, \bar{g})$ have
isometric tangent cones at $p_0$ and we construct a foliation of constant $H -
\bar{h}$ with prescribed angles $\bar{\gamma}$ near $p_0$. In fact, later in
Subsection \ref{strict barrier}, it is shown that this is the only case.

By the first variation formula of the mean curvatures
\begin{equation}
  H_{t, t^2 u} - H_t = - \Delta_t u - t^2
  (\ensuremath{\operatorname{Ric}}(N_t) + |A_t |^2) u + O (t), \label{mean
  curvature taylor expansion}
\end{equation}
where $\Delta_t$ is the Laplacian with respect to the induced rescaled metric
$t^{- 2} g|_{\Sigma_t}$. Note that $\ensuremath{\operatorname{Ric}} (N_t) = O
(t^{- 1})$ by the fact the tangent cone is $\mathrm{d} t^2 + a^2 t^2
g_{S^2}$. By the Taylor expansion of the function $\bar{h}$, we see
that
\begin{equation}
  \bar{h}_{t, t^2 u} - \bar{h}_t = \bar{h}' (t) t^2 u = t^2 u \nabla_{N_t}
  \bar{h} + O (t) . \label{prescribed function expansion}
\end{equation}
So
\begin{equation}
  (H_{t, t^2 u} - \bar{h}_{t, t^2 u}) - (H_t - \bar{h}_t) = - \Delta_t u - t^2
  (\ensuremath{\operatorname{Ric}}(N_t) + |A_t |^2 + \nabla_{N_t} \bar{h}) u +
  O (t) . \label{shear expansion}
\end{equation}
Note that both $H_t - \bar{h}_t$ and $H_{t, t^2 u} - \bar{h}_{t, t^2 u}$ are
finite and $|A_t |^2 + \nabla_{N_t} \bar{h} = O (t^{- 1})$ considering that
$(M, g)$ and $(M, \bar{g})$ have isometric tangent cones at $p_0$.

\begin{remark}
  \label{explain 1}We elaborate a bit more on \eqref{mean curvature taylor
  expansion} and its $O (t)$ remainder term. Since the metric $g$ is close to
  $\mathrm{d} t^2 + \psi (t)^2 g_{S^2}$ when $t \to 0^+$, we
  calculate the expansions with respect to the rescaled metric $t^{- 2} g$
  when computing for small $t > 0$. This is similar to
  {\cite{ye-foliation-1991}}. Then we rescale back and we obtain \eqref{mean
  curvature taylor expansion}. The term $O (t)$ involves products of $|A_t |$
  which is of order $t^{- 1}$ with terms of order at most $O (1)$. That is why
  the remainder is only of order $O (t)$ instead of $O (t^2)$.
\end{remark}

Also, the variation of angles give
\begin{align}
& t^{- 1} [\langle X_{t, t^2 u}, N_{t, t^2 u} \rangle - \langle X_t, N_t
\rangle] \\
= & - \sin \gamma \tfrac{\partial u}{\partial \nu_t} + t (- \cos \gamma A
(t^{- 1} \nu_t, t^{- 1} \nu_t) + A_{\partial M} (\eta_t {,} \eta_t)) u + O
(t^2), \label{angle variation}
\end{align}
where $\nu_t$ is the outward unit normal of $\partial \Sigma_t$ in $\Sigma_t$
with respect to the rescaled induced metric $t^{- 2} g|_{\Sigma_t}$ (note that
$t^{- 1} \nu_t$ is of unit length with respect to $g$). Other geometric
quantities are not rescaled. By the variation of the prescribed angle
$\bar{\gamma}$,
\begin{equation}
  t^{- 1} (\cos \bar{\gamma}_{t, t^2 u} - \cos \bar{\gamma}_t) = - \tfrac{t
  u}{\sin \bar{\gamma}} \partial_{\bar{\eta}_t} \cos \bar{\gamma} + O (t^2) .
  \label{background angle difference near pole} 
\end{equation}
So
\begin{align}
& t^{- 1} [(\langle X_{t, t^2 u}, N_{t, t^2 u} \rangle - \cos
\bar{\gamma}_{t, t^2 u}) - (\langle X_t, N_t \rangle - \cos \bar{\gamma}_t)]
\\
  = &- \sin \gamma \tfrac{\partial u}{\partial \nu_t}\\
  &\quad+ t (- \cos \gamma A
(t^{- 1} \nu_t, t^{- 1} \nu_t) + A_{\partial M} (\eta_t {,} \eta_t) +
\tfrac{1}{\sin \bar{\gamma}} \partial_{\bar{\eta}_t} \cos \bar{\gamma}) u +
O (t^2) . 
\label{angle deficit variation}
\end{align}
\begin{remark}
  \label{explain 2}The term $A (t^{- 1} \nu_t, t^{- 1} \nu_t) = O (t^{- 1})$,
  however, we observe that $\lim_{t \to 0} \bar{\gamma}_t = \pi / 2$, and
  $A_{\partial M} (\eta_t, \eta_t) = O (1)$ since $A_{\partial M}
  (\bar{\eta}_t, \bar{\eta}_t) = O (1)$. Or we can calculate with respect to
  the rescaling metric as in Remark \ref{explain 1}.
\end{remark}

Since $g$ and $\bar{g}$ has isometric tangent cone at $p_0$, we see that the
limit of the surface $(\Sigma_t, t^{- 2} g|_{\Sigma_t})$ as $t \to 0 $ is
$(\Sigma, a^2 g_{{S}^2})$ where $\Sigma$ is a scaling copy of a
geodesic disk of radius $\rho (0) = \lim_{t \to 0} \rho (t) > 0$ in the
standard 2-sphere. Consider the spaces
\[ \mathcal{Y}= \left\{ u \in C^{2, \alpha} (\Sigma) \cap C^{1, \alpha}
   (\bar{\Sigma}) : \int_{\Sigma} u = 0 \right\} \]
and
\begin{equation}
  \mathcal{Z}= \left\{ u \in C^{0, \alpha} (\Sigma) : \int_D u = 0 \right\} .
\end{equation}
Given small $\delta > 0$ and $\varepsilon > 0$, we define the map
\begin{equation}
  \Phi : (- \varepsilon, \varepsilon) \times B (0, \delta) \to \mathcal{Z}
  \times C^{1, \alpha} (\partial \Sigma)
\end{equation}
given by $\Phi (t, u) = (\Phi_1 (t, u), \Phi_2 (t, u))$ where $\Phi_i$, $i =
1, 2$ are given by
\begin{align}
\Phi_1 (t, u) = & (H_{t, t^2 u} - \bar{h}_{t, t^2 u}) - \frac{1}{| \Sigma |}
\int_{\Sigma} (H_{t, t^2 u} - \bar{h}_{t, t^2 u}), \\
\Phi_2 (t, u) = & t^{- 1} (\langle X_{t, t^2 u}, N_{t, t^2 u} \rangle - \cos
\bar{\gamma}_{t, t^2 u})
\end{align}
for $t \neq 0$. Here $B (0, \delta) \subset \mathcal{Y}$ is an open ball with
radius $\delta$ in the $C^{2, \alpha}$ norm and the integration on $\Sigma$ is
with respect to the metric $g_{\mathbb{S}^2}$. We extend $\Phi (t, u)$ to $t =
0$ by taking limits, that is,
\begin{equation}
  \Phi (0, u) = \lim_{t \to 0} \Psi (t, u) .
\end{equation}
We have the following proposition.

\begin{proposition}
  \label{foliation construction}For each $t \in [0, \varepsilon)$ with
  $\varepsilon$ small enough, we can find $u_t = u (\cdot, t) \in C^{2,
  \alpha} (\Sigma) \cap C^{1, \alpha} (\bar{\Sigma})$ such that $\int_{\Sigma}
  u (\cdot, t) = 0$ and
  \[ \Phi (t, u_t) = (0, 0) . \]
  In particular, each of the surfaces $\Sigma_{t, t^2 u}$ have constant
  $\lambda_t : = H_{t, t^2 u} - \bar{h}_{t, t^2 u}$ and prescribed angles
  $\gamma_{t, t^2 u} = \bar{\gamma}_{t, t^2 u}$. Moreover, $\lambda_t
  \leqslant 0$ for all small $t \in [0, \varepsilon)$.
\end{proposition}

Before proving this proposition, we give a variational lemma.

\begin{lemma}
  \label{differential schlaffli}Suppose that $(\Omega, \hat{g})$ is a compact
  manifold with piecewise smooth boundary $\partial \Omega$ and $\Sigma$ is a
  relatively open, smooth subset of $\partial \Omega$. Let $g_s$ be a smooth
  family of metrics indexed by $s \in [0, \varepsilon)$ such that $g_s \to
  \hat{g}$ as $s \to 0$, let $h_s = g_s - \hat{g}$. We now omit the subscript
  on $h_s$. Let $\nu$ be the unit outward normal of $\partial \Omega$ in
  $(\Omega, g)$, $H_g$ and $A_g$ be the mean curvatures and the second
  fundamental form of $\partial \Omega$ in $(\Omega, g)$ computed with respect
  to the unit normal pointing outward of $\Omega$, and $\gamma$ be the
  dihedral angles formed by $\Sigma$ and $\partial \Omega \backslash \Sigma$
  with respect to the metric $g$. We put a hat at appropriate places for the
  geometric quantities with respect to $\hat{g}$.
  
  Then
\begin{align}
& 2 \left[ - \int_{\Sigma} (H_g - H_{\hat{g}}) + \int_{\partial \Sigma}
\tfrac{1}{\sin \gamma_{\hat{g}}} (\cos \gamma_{\hat{g}} - \cos \gamma_g)
\right] \\
= & \int_{\Omega} ((R_g - R_{\hat{g}}) + \langle
\ensuremath{\operatorname{Ric}}_{\hat{g}}, h \rangle_{\hat{g}}) + 2
\int_{\partial \Omega \backslash \Sigma} (H_g - H_{\hat{g}}) +
\int_{\partial \Omega} \langle h, A_{\hat{g}} \rangle + O (s^2) .
\end{align}
  Here, we have used $O (s^2)$ to denote a remainder term comparable to
  $|h|_{\hat{g}}^2 + |h|_{\hat{g}} | \hat{\nabla} h|_{\hat{g}} + |
  \hat{\nabla} h|_{\hat{g}}^2$.
\end{lemma}

\begin{proof}
  From the variational formulas of the scalar curvature and the mean
  curvature, we have
  \begin{equation}
    R_g - R_{\hat{g}} = - \langle \ensuremath{\operatorname{Ric}}_{\hat{g}}, h
    \rangle_{\hat{g}} -\ensuremath{\operatorname{div}}_{\hat{g}} (\mathrm{d}
    (\ensuremath{\operatorname{tr}}_{\hat{g}} h)
    -\ensuremath{\operatorname{div}}_{\hat{g}} h) + O (s^2), \label{general
    scalar variation}
  \end{equation}
  and
  \begin{equation}
    2 (H_g - H_{\hat{g}}) = (\mathrm{d}
    (\ensuremath{\operatorname{tr}}_{\hat{g}} h)
    -\ensuremath{\operatorname{div}}_{\hat{g}} h) (\hat{\nu})
    -\ensuremath{\operatorname{div}}_{\sigma} Y - \langle h, A_{\hat{g}}
    \rangle_{\sigma} + O (s^2) \label{general mean variation}
  \end{equation}
  where $Y$ is the tangential component dual to the 1-form $h (\cdot,
  \hat{\nu})$. For the explicit form of the remainder terms, refer to
  {\cite[Proposition 4]{brendle-scalar-2011}} and {\cite{miao-mass-2021}}.
  
  We integrate the variation of the mean curvature \eqref{general mean
  variation} on the boundary $\partial \Omega$ with respect to the metric
  $\hat{g}$, we see
  \[ \int_{\partial \Omega} [(\mathrm{d}
     (\ensuremath{\operatorname{tr}}_{\hat{g}} h)
     -\ensuremath{\operatorname{div}}_{\hat{g}} h) (\hat{\nu})
     -\ensuremath{\operatorname{div}}_{\hat{\sigma}} Y - \langle h,
     A_{\hat{g}} \rangle] = 2 \int_{\partial \Omega} (H_g - H_{\hat{g}}) + O
     (s^2) . \]
  By the divergence theorem and the variation of the scalar curvature,
  \[ \int_{\partial \Omega} (\mathrm{d}
     (\ensuremath{\operatorname{tr}}_{\hat{g}} h)
     -\ensuremath{\operatorname{div}}_{\hat{g}} h) (\hat{g}) = \int_{\Omega}
     [- (R_g - R_{\hat{g}}) - \langle
     \ensuremath{\operatorname{Ric}}_{\hat{g}}, h \rangle_{\hat{g}}] + O (s^2)
     . \]
  For the term $\int_{\partial \Omega}
  \ensuremath{\operatorname{div}}_{\hat{\sigma}} Y$, we follow
  {\cite[(3.18)]{miao-mass-2021}} and obtain
  \[ \int_{\partial \Omega} \ensuremath{\operatorname{div}}_{\hat{\sigma}} Y =
     \int_{\Sigma} \ensuremath{\operatorname{div}}_{\hat{g}} Y +
     \int_{\partial \Omega \backslash \Sigma}
     \ensuremath{\operatorname{div}}_{\hat{\sigma}} Y = 2 \int_{\partial
     \Sigma} \tfrac{1}{\sin \hat{\gamma}} (\cos \hat{\gamma} - \cos \gamma) +
     O (s^2) . \]
  Collecting all the formulas in the proof proves the lemma.
\end{proof}

Lemma \ref{difference schlaffli} implies the following by taking the
difference of two families of metrics.

\begin{corollary}
  \label{difference schlaffli}Assume $(\Omega, \hat{g})$ is the manifold from
  Lemma \ref{difference schlaffli}, for two family of metrics $\{g_i \}_{i =
  1, 2}$ close to $\hat{g}$ indexed both by a small parameter $s$, we have
\begin{align}
& 2 \left[ - \int_{\Sigma} (H_{g_2} - H_{g_1}) + \int_{\partial \Sigma}
\tfrac{1}{\sin \hat{\gamma}} (\cos \gamma_{g_1} - \cos \gamma_{g_2})
\right] \\
= & \int_{\Omega} ((R_{g_2} - R_{g_1}) + \langle
\ensuremath{\operatorname{Ric}}_{\hat{g}}, g_2 - g_1 \rangle_{\hat{g}}) +
2 \int_{\partial \Omega \backslash \Sigma} (H_{g_2} - H_{g_1}) +
\int_{\partial \Omega} \langle g_2 - g_1, A_{\hat{g}} \rangle + O (s^2) .
\end{align}
\end{corollary}

Now we are ready to prove Proposition \ref{foliation construction}.

\begin{proof}[Proof of Proposition \ref{foliation construction}]
  The proof is similar to {\cite{chai-scalar-2023-arxiv}}. We bring up only
  the main differences.
  
  Because the right hand of both \eqref{shear expansion} and \eqref{angle
  deficit variation} converge to $\Delta u$ and $\tfrac{\partial u}{\partial
  \nu}$ (up to a constant) respectively, so we can first follow
  {\cite[Proposition 4.2]{chai-scalar-2023-arxiv}} to construct a foliation
  $\{\Sigma_{t, t^2 u} \}_{t \in [0, \varepsilon)}$ near $p_0$ with constant
  $H - \bar{h}$ and $\gamma_{t, t^2 u} = \bar{\gamma}_{t, t^2 u}$ along
  $\partial \Sigma_{t, t^2 u}$, and then {\cite[Lemma
  4.3]{chai-scalar-2023-arxiv}} to obtain that
  \begin{equation}
    - \lambda_t | \Sigma_t | = \int_{\Sigma_t} (H_t - \bar{h}_t) +
    \int_{\partial \Sigma_t} \tfrac{1}{\sin \gamma_t} (\cos \bar{\gamma}_t -
    \cos \gamma_t) + O (t^3) . \label{averaged difference}
  \end{equation}
  Now we show that $\lim_{t \to 0} \lambda_t \leqslant 0$.
  
  We consider the rescaled set $t^{- 1} \Omega_t$ with two rescaled metrics
  $t^{- 2} g$ and $t^{- 2} \bar{g}$. Since $\bar{g} = \mathrm{d} t^2 + \psi
  (t)^2 g_{{S}^2}$ and $\psi (t) = a t + o (t)$, it is easy to see that
  $(t^{- 1} \Omega_t, t^{- 2} \bar{g})$ converges to a truncated metric cone
  $\Lambda = (0, 1] \times D$ with the metric $\varrho := \mathrm{d} s^2 + a^2
  s^2 g_{{S}^2}$ where $s \in (0, 1]$ and $(D, a^2 g_{{S}^2})$
  is some convex disk in a 2-sphere $({S}^2, a^2 g_{{S}^2})$.
We set $D_s = \{s\} \times D$. Since
  $g$ and $\bar{g}$ has isometric tangent cone at $p_0$, $(t^{- 1} \Omega_t,
  t^{- 2} g)$ converges to $(\Lambda, \varrho)$ as well. Therefore, we can
  view $g_1 = t^{- 2} g$ and $g_2 = t^{- 2} \bar{g}$ (indexed by $t$) as two
  metrics on $\Lambda$ getting closer to $\varrho$ as $t \to 0$. We rescale
  \eqref{averaged difference} by a factor of $t^{- 2}$, we obtain
  \[ - \lambda_t | \Sigma_t | t^{- 2} = \int_{\Sigma_t} (H_t - \bar{h}_t) t^{-
     2} + \int_{\partial \Sigma_t} \tfrac{1}{\sin \gamma_t} (\cos
     \bar{\gamma}_t - \cos \gamma_t) t^{- 2} + O (t) \]
  which is equivalent to
  \[ - \lambda_t |D|_{g_1} = \int_D (H_{g_2} - H_{g_1}) + \int_{\partial D}
     \tfrac{1}{\sin \gamma_t} (\cos \bar{\gamma}_t - \cos \gamma_t) + O (t) .
  \]
  In the above the integration done is with respect to the metric $g_1$ and
  $H_{g_i}$ are the mean curvature of $\{1\} \times D$ in $(\Lambda, g_i)$
  computed with respect to the normal pointing inside of $\Lambda$.
  
  All the comparisons in item (\ref{item metric cone}) of Theorem \ref{rigidity depending on structures} carry over to the rescaled
  metrics $g_1$ and $g_2$ on $\Lambda$, and that $(\Lambda,
  \varrho)$ has non-negative Ricci curvature by the assumptions of item (\ref{item metric cone}) of Theorem \ref{rigidity depending on structures}. We use Corollary
  \ref{difference schlaffli} and arrive that $\lambda_t \leqslant O (t)$, that
  is,
  \[ \lim_{t \to 0} \lambda (t) \leqslant 0. \]
  Since $\lambda_t$ satisfies the differential inequality \eqref{original ode}
  and considering the asymptotics $u (\cdot, t) = 1 + O (t)$, $\cot
  \bar{\gamma} = O (t)$ and $\bar{h} =  2/t  + O (1)$ in
  \eqref{Psi}, we see that $\lambda_t \leqslant 0$ for all $t \in (0,
  \varepsilon)$.
\end{proof}

\begin{remark}
  The Ricci curvature in Corollary \ref{difference schlaffli} blows up near
  $\{0\} \times D$, however, because we are integrating with respect to the
  metric $\varrho$, the volume near $\{0\} \times D$ is small. Also, the
  difference $g_2 - g_1$ is small. So the blowing up of the Ricci curvature
  will not cause an issue.
\end{remark}

\subsection{Barrier construction with non-isometric tangent
cones}\label{strict barrier}

Since $\bar{g} = \mathrm{d} t^2 + \psi (t)^2 g_{{S}^2}$, the manifold
$(M, \bar{g})$ is topologically a cone near $t = 0$ and it is a point at $t =
0$. According to the assumptions of item (\ref{item metric cone}) of Theorem \ref{rigidity depending on structures}, $(M, g)$ at $p_0$ also
locally resembles a cone, that is,
\begin{equation}
  g = \mathrm{d} s^2 + s^2 g_0 + g_1, \label{metric expansion of g}
\end{equation}
where $s$ is a parameter, $g_0$ is a metric on a two dimensional disk $D$ and
$g_1$ is small compare to $\mathrm{d} s^2 + s^2 g_0$. In other words, the
tangent cone at $p_0$ is a cone with the metric $\mathrm{d} s^2 + s^2 g_0$.

Now we can also identify $M$ near $p_0$ as $(0, \varepsilon) \times D$ and $t$
as a function on $(0, \varepsilon) \times D$. Let $(s, x) \in (0, \varepsilon)
\times D$, we see that $\tau : = s / t$ as a function on $M$ only depends on
$x \in D$. So we view $\tau$ as a function on $D$. Since $g \geqslant \bar{g}$
on $M$, we have that $\tau (x) \geqslant 1$. Now we discuss the case that
$\tau (x) \equiv 1$ on $D$.

\begin{lemma}
  \label{tau 1 imply isometric cone}If $\tau \equiv 1$ on $D$, then $g_0 = a^2
  g_{{S}^2}$. That is, $(M, g)$ and $(M, \bar{g})$ have isometric
  tangent cones at $p_0$.
\end{lemma}

\begin{proof}
  Since $\tau \equiv 1$, so we can rescale $(M, \bar{g})$ and $(M, g)$ by the
  same scale to obtain a cone $\mathcal{C}= (0, \infty) \times D$ but with two
  different metrics $\chi_1 = \mathrm{d} t^2 + a^2 t^2 g_{{S}^2}$ and
  $\chi_2 = \mathrm{d} t^2 + t^2 g_0$. For $s > 0$, set $D_s = \{s\} \times D
  \subset \mathcal{C}$. Since the metric comparison, the mean curvature and
  the scalar curvature comparison are preserved by rescaling, so $g_0
  \geqslant a^2 g_{{S}^2}$, the scalar curvature $R_{\chi_2} \geqslant
  R_{\chi_1}$ and the mean curvature of $\partial \mathcal{C}$ at $\partial
  D_1$ satisfies $H_{\chi_2} \geqslant H_{\chi_1}$.
  
  Since both $\chi_i$, $i = 1, 2$ are warped product metrics, the comparison
  $R_{\chi_2} \geqslant R_{\chi_1}$ reduces to Gaussian curvature comparison
  $K_2 \geqslant K_1 = a^{- 2}$ of $(D_1, g_0)$ and $(D_1, a^2
  g_{{S}^2})$ by a direct computation of scalar curvature (or Gauss
  equation). Let $\kappa_i$ be the geodesic curvatures of $\partial D_1$ with
  respect to $\chi_i |_{D_1}$. By direct calculation, the second fundamental
  form $A^{(i)}_{\partial \mathcal{C}}$ of $\partial \mathcal{C}$ in the
  direction $\partial_t$ vanishes with respect to both metrics $\chi_i$ and
  the second fundamental form $A_{D_1}^{(i)}$ of $D_1$ in $\mathcal{C}$ with
  respect to $\chi_i$ agree. It then follows from $H_{\chi_2} \geqslant
  H_{\chi_1}$ and \eqref{rs rewrite} that $\kappa_2 \geqslant \kappa_1$.
  
  To summarize, we have comparisons on $D_1$ that $g_0 \geqslant a^2
  g_{{S}^2}$, $K_2 \geqslant K_1$ and $\kappa_2 \geqslant \kappa_1$
  along $\partial D_1$. By Gauss-Bonnet theorem, $g_0 \equiv a^2
  g_{{S}^2}$ on $D_1$ and it follows that $\chi_1 \equiv \chi_2$.
  Therefore, $(M, g)$ and $(M, \bar{g})$ have isometric tangent cones at
  $p_0$.
\end{proof}

By the above lemma, the case $\tau \equiv 1$ is the case which implies
isometric tangent cones of $(M, g)$ and $(M, \bar{g})$ at $p_0$. This is the
case we have already addressed in Subsection \ref{foliation}. Without loss of
generality, we assume that $\tau \mathrel{\not\equiv} 1$.

We first consider the difference of $H - \bar{h}$ of the perturbation for
$D_s$. We now represent $\bar{h}$ at $D_s$ and its value at the graphical
perturbations of $D_s$ by $\zeta$ to avoid notational confusion. By the first
variation of the mean curvatures,
\begin{align}
& (H_{s, s^2 u} - \zeta_{s, s^2 u}) - (H_s - \zeta_s) \\
= & -\Delta_s u - s^2 (\ensuremath{\operatorname{Ric}}(N_s) + |A_s |^2 + s^{-
2} (\zeta_{s, s^2 u} - \zeta_s)) u + O (s),
\end{align}
where $\Delta_s$ is the Laplacian with respect to the metric $s^{- 2}
g|_{D_s}$.

\begin{remark}
  We have $\{(s^{- 1} D_s, s^{- 2} g|_{D_s})\}_{s > 0}$ converges to $(D,
  g_0)$ as $s \to 0$ by the metric \eqref{metric expansion of g} near $p_0$,
  and to indicate that the limit carries the metric $g_0$, we use $D_0$ instead
  of $D$ only.
\end{remark}

\begin{lemma}
  \label{favorable sign}We have that
  \[ s^2 (|A_s |^2 - s^{- 2} (\zeta_{s, s^2 u} - \zeta_s)) = (2 - 2 \tau) + O
     (s) . \]
\end{lemma}

\begin{proof}
  Since $\{(s^{- 1} \Lambda_s, s^{- 2} g)\}_{s > 0}$ converges to a truncated
  radial cone and $\{(s^{- 1} D_s, s^{- 2} g|_{D_s})\}_{s > 0}$ converges to
  the section of the radial cone with unit distance to $p_0$, so the section
  has second fundamental form $- 2$ and by rescaling,
  \[ |A_s |^2 = 2 s^{- 2} + O (s^{- 1}) \]
  as $s \to 0$.
  
  At a point $p = (s, x) \in D_s$, the value of $t$ is given by $t = s \tau
  (x)$ where $x$ is the projection of $p$ to the second coordinate. Since
  $\tau$ as a function on $M$ only depends on $x$, we see that the value of
  the function $t$ at the graphical perturbation $s + s^2 u$ of $D_s$ is given
  by $(s + s^2 u) \tau$. Since $\bar{h} (t) =  2 t^{- 1} + O (1)$, so
  \[ \zeta_{s, s^2 u} - \zeta_s =  \tfrac{2}{(s + s^2 u) \tau} - \tfrac{2}{s
     \tau} + O (1) = -\tfrac{2 \tau}{s^2} (s^2 u) + O (1) . \]
  Hence
  \[ s^2 (|A_s |^2 + s^{- 2} (\zeta_{s, s^2 u} - \zeta_s)) = (2 - 2 \tau) + O
     (s), \]
  which proves the lemma.
\end{proof}

Let $f = \lim_{s \to 0} s^2 (\ensuremath{\operatorname{Ric}}(N_s) + |A_s |^2 +
s^{- 2} (\zeta_{s, s^2 u} - \zeta_s))$ which is a function on the limit $D_0$,
so
\begin{equation}
  \lim_{s \to 0} [(H_{s, s^2 u} - \zeta_{s, s^2 u}) - (H_s - \zeta_s)] =
  -\Delta_0 u - f u, \label{mean curvature difference limit}
\end{equation}
where $\Delta_0$ is the Laplacian of $D_0$. Recall that
$\ensuremath{\operatorname{Ric}} (N_s) = O (s^{- 1})$, so
\[ f = 2 - 2 \tau \text{ on } D_0 . \]
Let $\alpha_s$ be the dihedral angles formed by $\partial M$ and $D_s$, and
$\alpha_{s, s^2 u}$ be the angles formed by $\partial M$ and the graphical
perturbation of $D_s$.

\begin{lemma}
  \label{angle to pi/2}The dihedral angles $\alpha_s$ formed by $\partial M$
  and $D_s$ approach $\pi / 2$ as $s \to 0$.
\end{lemma}

\begin{proof}
  Since $\{(s^{- 1} \Lambda_s, s^{- 2} g)\}_{s > 0}$ converges to a truncated
  radial cone, $\{(s^{- 1} D_s, s^{- 2} g|_{D_s})\}_{s > 0}$ converges to the
  section of the radial cone with unit distance to $p_0$, and this section is
  orthogonal to the radial direction in the limit, \ so the intersection
  angles of $\partial M$ and $D_s$ approaches $\pi / 2$ as $s \to 0$.
\end{proof}

\begin{lemma}
  \label{boundary 2ff at eta}We have that $A_{\partial M} (\eta, \eta) = O
  (1)$.
\end{lemma}

\begin{proof}
  The lemma can be deduced from that $\eta$ is approximately the radial
  direction $\partial_s$ as $s \to 0$, the scaling property of $A_{\partial
  M}$ and the following lemma.
\end{proof}

\begin{lemma}
  Let $(S, \sigma)$ be a 2-surface with boundary and $(C = [0, \infty) \times
  S, \mathrm{d} s^2 + s^2 \sigma)$ be the cone over $(S, \sigma)$. Then the
  second fundamental form of $\partial C$ in $C$ in the direction $\partial_t$
  vanishes.
\end{lemma}

\begin{proof}
  Let $Z$ be a tangent vector field over $\Sigma$, then by direct calculation
  $\nabla_{\partial_t} Z = \nabla_X \partial_t = s^{- 1} Z$. So $\langle
  \nabla_{\partial_t} Z, \partial_t \rangle = 0$ since on $C$ the metric is
  $\mathrm{d} t^2 + t^2 \sigma$. Due to the same reason, the unit normal
  vector $Z$ of $\partial C$ in $M$ is tangent to $\Sigma$, so the claim is
  proved.
\end{proof}

We are interested in the difference between $\alpha_{s, s^2 u}$ and the value
of $\bar{\gamma}$ which to avoid confusion we denote by $\beta_s$ ($\beta_{s,
s^2 u}$) at (the graphical perturbation $s^2 u$ of) $D_s$. Using the relation
of $s$ and $t$, $\beta = \bar{\gamma}_{s / \tau, s^2 u / \tau}$. By the
expansion of angles (see \eqref{angle variation}), we see
\[ \cos \alpha_{s, s^2 u} - \cos \alpha_s = -\sin \alpha_s \tfrac{\partial
   u}{\partial \nu_s} + s (- \cos \alpha_s A (s^{- 1} \nu_s, s^{- 1} \nu_s) +
   A_{\partial M} (\eta_s {,} \eta_s)) u + O (s^2) . \]
And
\[ s^{- 1} (\cos \beta_{s, s^2 u} - \cos \beta_s) = s u \tau^{- 1}
   \nabla_{\eta_{s / \tau}} \cos \bar{\gamma}_{s / \tau, s^2 u / \tau} + O
   (s^2) \]
Since each $\Sigma_t$ is stable capillary minimal surface under the metric $\bar{g}$, so we know that
\[  \tfrac{1}{\sin \bar{\gamma}} \nabla_{\eta_t} \cos \bar{\gamma} = - \cos
   \bar{\gamma} A (\nu_t, \nu_t) + A_{\partial M} (\eta_t {,} \eta_t) . \]
Based on the above asymptotic analysis and Lemmas \ref{angle to pi/2} and
\ref{boundary 2ff at eta}, we see
\begin{equation}
  \lim_{s \to 0} [s^{- 1} (\cos \alpha_{s, s^2 u} - \cos \alpha_s) - s^{- 1}
  (\cos \beta_{s, s^2 u} - \cos \beta_s)] =- \tfrac{\partial u}{\partial \nu_0}
  \label{angle difference limit}
\end{equation}
on $\partial D_0$ where $\nu_0$ is the outward normal of $\partial D_0$ in
$D_0$. By the elliptic strong maximum principle, the operator
\[ (-\Delta_0 - f,-\tfrac{\partial}{\partial \nu_0}) : C^{2, \alpha} (D_0) \cap
   C^{1, \alpha} (\bar{D}_0) \to C^{0, \alpha} (D_0) \times C^{0, \alpha}
   (\partial D_0) \]
is an isomorphism since $f \leqslant 0$ in $D_0$ due to Lemma \ref{favorable
sign} and $\tau \gneqq 1$. In other words, we can specify the limits
\begin{align}
& \lim_{s \to 0} [(H_{s, s^2 u} - \zeta_{s, s^2 u}) - (H_s - \zeta_s)]
\\
& \quad \text{ and } \lim_{s \to 0} [s^{- 1} (\cos \alpha_{s, s^2 u} - \cos
\alpha_s) - s^{- 1} (\cos \beta_{s, s^2 u} - \cos \beta_s)]
\end{align}
by choosing a suitable $u \in C^{2, \alpha} (D_0) \cap C^{1, \alpha}
(\bar{D}_0)$.

We have these facts: by Lemma \ref{angle to pi/2}, both $\alpha_s$ and
$\beta_s$ tend to $\pi / 2$ as $s \to 0$, so $\lim_{s \to 0} s^{- 1} (\alpha_s
- \beta_s)$ is a function on $\partial D_0$; \( H_s - \zeta_s = (2 - 2 \tau) s^{- 1} + O (1)\);
\begin{equation}\label{perturbation expansion}
  \text{ } H_{s, s^2 u} -
  \zeta_{s, s^2 u} = (2 - 2 \tau) s^{- 1} + O (1)
  \end{equation}
 for small $s > 0$ with a remainder term depending on $u$.
 Hence, we can specify a function $u$ to counter-effect the
$O (1)$ remainder term in $H_{s} - \zeta_{s}$ and make the remainder term in \eqref{perturbation expansion} strictly negative. That is, we can
specify a function $u$ such that
\begin{align}
\lim_{s \to 0} (H_{s, s^2 u} - \zeta_{s, s^2 u} - (2 - 2 \tau) s^{- 1}) & =
u_0 \text{ in } D_0, \\
\quad \lim_{s \to 0} s^{- 1} (\cos \alpha_{s, s^2 u} - \cos \beta_{s, s^2
u}) & < 0 \text{ along } \partial D_0,
\end{align}
for some negative function $u_0 \in C^{0, \alpha} (\bar{D}_0)$. Recall the
definitions of $\zeta$, $\tau$, $\beta$, and by continuity, there exists a
surface $\Sigma_- \subset M$ satisfying
\[ H - \bar{h} < 0 \text{ in } \Sigma_- \text{ and } \alpha > \bar{\gamma}
   \text{ along } \partial \Sigma_- . \]
This surface $\Sigma_-$ is a lower barrier in the sense of Definition
\ref{barrier condition}.

Now we can prove item (\ref{item metric cone}) of Theorem \ref{rigidity depending on structures}.

\begin{proof}[Proof of item (\ref{item metric cone}) of Theorem \ref{rigidity depending on structures}]
  Assume that $g$ and $\bar{g}$ do not have isometric tangent cone at $p_0$,
  then we can construct a barrier $\Sigma_-$ such that $H - \bar{h} < 0$ in
  $\Sigma_-$ and the angle $\alpha > \bar{\gamma}$ along $\partial \Sigma_-$.
  But due to item (\ref{item simple}) of Theorem \ref{rigidity depending on structures},
  this is not possible. So $g$ and $\bar{g}$ have isometric tangent cones at
  $p_0$, then by the construction of the foliation in Theorem \ref{foliation
  construction}, again we have a barrier near $t = 0$, but the barrier
  condition is now not strict. We can extend the rigidity $g = \bar{g}$ in item (\ref{item simple}) of Theorem \ref{rigidity depending on structures}
  beyond the barrier and to all of $M$.
\end{proof}

\begin{remark}
  \label{alternative to llarull}By considering only the mean curvature, this
  provides an alternative proof of Theorem \ref{llarull} in dimension 3.
  Moreover, we allow conical metrics of $(\mathbb{S}^3, g)$ at two antipodal
  points.
\end{remark}

\begin{remark}
  \label{gb in barrier}During the construction of barriers in the case of non-isometric cones, the Gauss-Bonnet
  theorem is only used in Lemma \ref{tau 1 imply isometric cone}.
\end{remark}

\section{Construction of barriers (II)}\label{barrier II}

In this section, we prove items (\ref{item euclid cone}) and (\ref{smooth}) of Theorem \ref{rigidity depending on structures}. Our method is similar
to the previous work {\cite{chai-scalar-2023-arxiv}}.

\subsection{Proof of item (\ref{item euclid cone}) of Theorem \ref{rigidity depending on structures}}
For convenience, we set $t_- = 0$. We will construct a lower barrier near $t=0$. As before, for any $t>0$, we set $\Sigma_t$ to be the $t$-level set of $t$ and $\Omega_t$ to be the $t$-sublevel set, that is, all points of $M$ which lie below $\Sigma_t$. 
We see from the assumption on the tangent cone at $p_-$ that the sequence $\{(t^{- 1} M, t^{- 2} \bar{g})\}_{t > 0}$ converges to some
right circular cone $\bar{C}$ in $\mathbb{R}^3$ equipped with a flat metric
$g_{\mathbb{R}^3}$ as $t \to 0$. Then $\{(t^{- 1} M, t^{- 2} g)\}_{t > 0}$
converges to the same cone $\bar{C}$ but with a different constant metric
$g_0$. The cone $(\bar{C}, g_0)$ can be represented as a cone in the Euclidean space $(\mathbb{R}^3, g_{\mathbb{R}^3})$.

%
We have the existence of a barrier if $(M, g)$ and $(M, \bar{g})$ have
non-isometric tangent cones at $p_0$.

\begin{lemma}
  \label{barrier with cone angle comparison}Let $M$ be given as in item (\ref{item euclid cone}) of Theorem \ref{rigidity depending on structures}. If the tangent cones
  of $(M, g)$ and $(M, \bar{g})$ at $p_-$ are not isometric, assume that the
  mean curvature comparison and the metric comparison hold near $p_-$, then
  there exists a surface $\Sigma_{-}$ satisfying
  \[ H - \bar{h} < 0 \text{ in } \Sigma_{-} \text{ and } \gamma_{\Sigma_-} > \bar{\gamma} \text{
     along } \partial \Sigma_{-} \]
  as the above. This surface $\Sigma_{-}$ is a barrier in the sense of Definition
  \ref{barrier condition}.
\end{lemma}

\begin{proof}
  First, we note that the mean curvature comparison and the metric comparison
  (we only need boundary metric comparison) are preserved in the limits. Because that the tangent cone of $(M,\bar g)$ at $p_-$ is a round (solid) circular cone and that $(t^{-1}\Sigma_t,t^{-2} \bar{g})$ converges to its axial section as $t\to 0$, it follows from the angle comparison of
  {\cite[Proposition 4.9]{chai-scalar-2023-arxiv}} that there exists a plane $P$
  in $\bar{C}$ such that the the dihedral angles formed by $\partial \bar{C}$
  and $P$ in the metric $g_0$ are everywhere larger than $\bar{\gamma} (t_-)$.
  
  We gain a lot of freedom to construct the barrier from the
  \text{{\itshape{strict}}} comparison of angles. The rest of the argument is
 analogous to {\cite[Proposition 4.10]{chai-scalar-2023-arxiv}}.
\end{proof}

\begin{remark}
  Note that the scalar curvature comparison is not needed here.
\end{remark}

\begin{proof}[Proof of item (\ref{item euclid cone}) of Theorem \ref{rigidity depending on structures}]
  First, the tangent cones of $(M, g)$ and $(M, \bar{g})$ at $p_-$ must be
  isometric. Indeed, by Lemma \ref{barrier with cone angle comparison} and
  item (\ref{item simple}) of Theorem \ref{rigidity depending on structures}, the barrier constructed in Lemma \ref{barrier with
  cone angle comparison} cannot have $H - \bar{h} < 0$ in $\Sigma_{-}$ and $\gamma_{\Sigma_-} <
  \bar{\gamma}$ hold strictly along $\partial\Sigma_{-}$.
  
  By following Subsection \ref{foliation}, we can construct graphical
  perturbations $\Sigma_{t, t^2 u}$ of $\Sigma_t$ which satisfy Proposition
  \ref{foliation construction}. For every sufficiently small $t > 0$,
  $\Sigma_{t, t^2 u}$ \ is a barrier in the sense of Definition \ref{barrier
  condition}, we conclude that $g = \bar{g}$ for the region bounded by
  $\Sigma_{t, t^2 u}$ and $P_+ \cap \partial M$ for every $t > 0$ from item (\ref{item simple}) of Theorem \ref{rigidity depending on structures}. Hence, we finish the proof of item (\ref{item euclid cone}) of Theorem \ref{rigidity depending on structures}.
\end{proof}

%
%
\subsection{Proof of item (\ref{smooth}) of Theorem \ref{rigidity depending on structures}} \label{sec:spherical_part}

This part is a slightly extension of the argument in Section 5 in our previous paper \cite{chai-scalar-2023-arxiv} and Ko-Yao's paper \cite{ko-scalar-2024}.
So we only sketch the key steps here and refer to the above papers for more details.

Again, we set $s_- = 0$ and we construct a lower barrier near $s=0$. Suppose $M$ is given by
\[
	M=\{ (s,p) \in [0,\varepsilon)\times {S}^2:s\ge f(p) \},
\]
near $p_-=O$, where $f$ is a smooth function such that $f(p_-)=0$ and $\mathrm{Hess}f$ is positive definite at $p_-$ under metric $g_{S^2}$.
Note that we need to assume $\psi(0)\neq 0$, otherwise, the manifold $M$ will have a cusp at point $O=(0,0,0)$.
For simplicity, we assume $\psi(t(0))=1$, where $t(s)$ is determined by \eqref{conformally related metric}.






To better illustrate the situation, we can choose the coordinate $(x_1,x_2)$ on ${S}^2$ such that the expansion of metric $\bar{g}$ in \eqref{conformally related metric} at $O$ is given by
\[
	\bar{g}=ds^2+dx_1^2+dx_2^2+O(s)+O(|x|^2),
\]
where $|x|=\sqrt{x_1^2+x_2^2}$.

For simplicity, we denote $\bar{g}_0=ds^2+dx_1^2+dx_2^2$ as the linearised part of $\bar{g}$ at $O$.
After a suitable rotation, we can write $g=g_0+sh+O(s^2)$ for some constant metric $g_0$ defined as
\begin{equation}
	g_0=a_{33}ds^2+(a_{11}dx_1^2+a_{22}dx_2^2)+2a_{13}dx_1ds+2a_{23}dx_2ds,
	\label{eq:G0Def}
\end{equation}
where the matrix
\[
	\begin{bmatrix}
		a_{11} & 0 & a_{13}\\
		0 & a_{22} & a_{23} \\ 
		a_{13} & a_{23} & a_{33}
	\end{bmatrix}
\]
is positive definite and satisfies $a_{11},a_{22},a_{33}\ge 1$.

We assume the manifold $M$ can be written as
\[
	M=\{ (s,x_1,x_2):s \in  [0,\varepsilon), s\le \zeta(x_1,x_2) \},
\]
where $\zeta(x_1,x_2)=c_{11}x_1^2+2c_{12}x_1x_2+c_{22}x_2^2+O(|x|^3)$ is a smooth function with $c_{11},c_{22},c_{11}c_{22}-c_{12}^2>0$.
Here, we have used the fact that $\mathrm{Hess}f$ is positive definite at $p_-$ under metric $g_{S^2}$.


We write $a^{ij}$ as the inverse matrix of $a_{ij}$, and define several constants as
\begin{align*}
	B={} & \sqrt{a^{33}((\sqrt{a_{11}}c_{22}+\sqrt{a_{22}}c_{11})^2+(\sqrt{a_{11}}-\sqrt{a_{22}})^2c_{12}^2)}\\
	b_{11}={}& a^{33}B^{-1}c_{11}^{-1}(a_{11}(c_{11}c_{22}-c_{12}^2)+\sqrt{a_{11}a_{22}}(c_{11}^2+c_{12}^2))\\
	b_{12}={}& b_{21}= a^{33}B^{-1} \sqrt{a_{11}a_{22}}(c_{11}+c_{22})\\
	b_{22}={}& a^{33}B^{-1}c_{22}^{-1}(a_{22}(c_{11}c_{22}-c_{12}^2)+\sqrt{a_{11}a_{22}}(c_{12}^2+c_{22}^2)).
\end{align*}
and consider the function $G_{\lambda,s}$ defined by
\begin{align*}
	G_{\lambda,s}(x_1,x_2)={}&c_{11}(b_{11}(1+\lambda)-1)x_1^2+c_{22}(b_{22}(1+\lambda)-1)x_2^2\\
					   & + 2c_{12}(b_{12}(1+\lambda)-1)x_1x_2-s^2.
\end{align*}
and the surface $\Sigma_{\lambda,s}$ is defined by
\[
	\Sigma_{\lambda,s}=\left\{ (G_{\lambda,s}(x), x): x \in \mathbb{R}^2 \text{ and } G_{\lambda,s}(x)\ge \zeta(|x|^2) \right\}.
\]
We use an ellipse $E_\lambda$ to parameterize $\Sigma_{\lambda,s}$ where $E_\lambda\subset \mathbb{R}^2$ is given by
\[
	E_\lambda:=\{ \hat{x} \in \mathbb{R}^2: c_{11}b_{11}\hat{x}_1^2+c_{22}b_{22}\hat{x}_2^2+2c_{12}b_{12}\hat{x}_1\hat{x}_2< \frac{1}{1+\lambda}\}.
\]

Then, the surface $\Sigma_{\lambda,s}$ can be written as a map $E_\lambda\rightarrow \Sigma_{\lambda,s}$ such that
\[
	\Sigma_{\lambda,s}(\hat{x}):=(G_{\lambda,s}(\Phi_{\lambda,s}(\hat{x})),\Phi_{\lambda,s}(\hat{x}))
\]
where $\Phi_{\lambda,s}:E_\lambda\rightarrow \mathbb{R}^2$ satisfies
\[
	\Phi_{\lambda,s}(\hat{x})=s\hat{x}+O(s^3).
\]

We also use $\Sigma_s=\Sigma_{0,s}$ for short.
We have the following result by the argument in \cite{chai-scalar-2023-arxiv}.

\begin{figure}[ht]
    \centering
	\begingroup
	\def\svgwidth{0.8\columnwidth}
\begingroup%
  \makeatletter%
  \providecommand\color[2][]{%
    \errmessage{(Inkscape) Color is used for the text in Inkscape, but the package 'color.sty' is not loaded}%
    \renewcommand\color[2][]{}%
  }%
  \providecommand\transparent[1]{%
    \errmessage{(Inkscape) Transparency is used (non-zero) for the text in Inkscape, but the package 'transparent.sty' is not loaded}%
    \renewcommand\transparent[1]{}%
  }%
  \providecommand\rotatebox[2]{#2}%
  \newcommand*\fsize{\dimexpr\f@size pt\relax}%
  \newcommand*\lineheight[1]{\fontsize{\fsize}{#1\fsize}\selectfont}%
  \ifx\svgwidth\undefined%
    \setlength{\unitlength}{680.31496063bp}%
    \ifx\svgscale\undefined%
      \relax%
    \else%
      \setlength{\unitlength}{\unitlength * \real{\svgscale}}%
    \fi%
  \else%
    \setlength{\unitlength}{\svgwidth}%
  \fi%
  \global\let\svgwidth\undefined%
  \global\let\svgscale\undefined%
  \makeatother%
  \begin{picture}(1,0.5)%
    \lineheight{1}%
    \setlength\tabcolsep{0pt}%
    \put(0,0){\includegraphics[width=\unitlength,page=1]{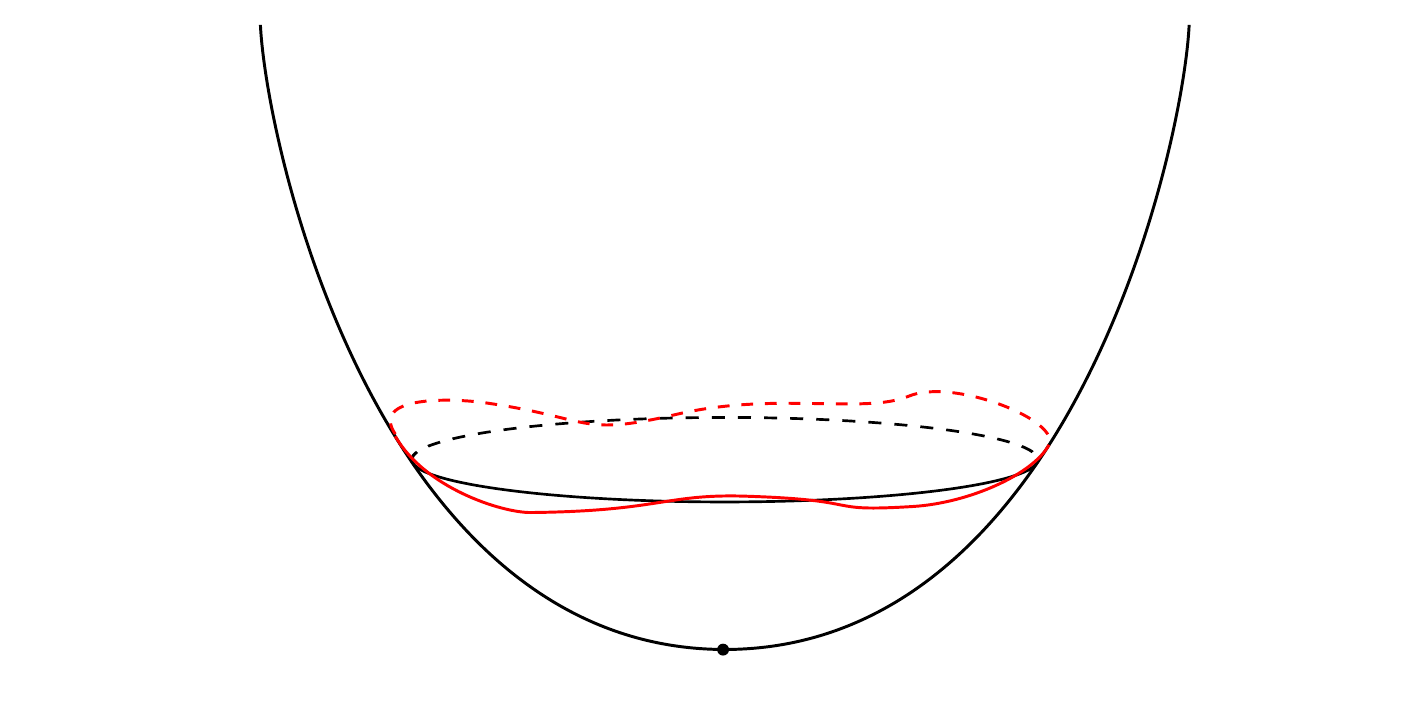}}%
    \put(0.73644702,0.15431277){\color[rgb]{0,0,0}\makebox(0,0)[lt]{\lineheight{0}\smash{\begin{tabular}[t]{l}$\Sigma_{s'}$\end{tabular}}}}%
    \put(0,0){\includegraphics[width=\unitlength,page=2]{spherical-point.pdf}}%
    \put(0.81481101,0.33425357){\color[rgb]{1,0,0}\makebox(0,0)[lt]{\lineheight{0}\smash{\begin{tabular}[t]{l}$\Sigma_{\lambda,s}$\end{tabular}}}}%
    \put(0.81040338,0.30223726){\color[rgb]{0,0,0}\makebox(0,0)[lt]{\lineheight{0}\smash{\begin{tabular}[t]{l}$\Sigma_s$\end{tabular}}}}%
    \put(0,0){\includegraphics[width=\unitlength,page=3]{spherical-point.pdf}}%
    \put(0.75675271,0.18785998){\color[rgb]{1,0,0}\makebox(0,0)[lt]{\lineheight{0}\smash{\begin{tabular}[t]{l}$\Sigma_{s',u}$\end{tabular}}}}%
  \end{picture}%
\endgroup%

	\endgroup

    \caption{Construction of $\Sigma_{\lambda,s}$ and $\Sigma_{s,u}$.}
    \label{fig:spherical-point}
\end{figure}

\begin{proposition}
	\label{prop_angleAsy}
	Suppose the metric $g$ can be written as $g=g_0+sh+O(s^2)$ where $g_0$ is the constant metric defined in \eqref{eq:G0Def} and $h$ is a bounded symmetric two-tensor.
	Then, we have
	\begin{align*}
		\cos \gamma_{\lambda,s}(\hat{x})={}&\cos \bar{\gamma}_{\lambda,s}(\hat{x})
		-4\lambda s^2\left[ 
		\frac{(\hat{x}_1c_{11}b_{11}+\hat{x}_2c_{12}b_{12})^2}{a_{11}a^{33}}+\frac{(\hat{x}_1c_{12}b_{12}+\hat{x}_2c_{22}b_{22})^2}{a_{22}a^{33}}\right] \\
		{}&+s^2O(\lambda^2)+A(\hat{x})s^3+L(h)s^3+O(s^4).
	\end{align*}
	for any $\hat{x} \in E_\lambda$.
	Here, $A(\hat{x})$ is a bounded term (not related to $s$ and $h$) which is also odd symmetric with respect to $\hat{x}$, $L(h)$ is a bounded term (not related to $s$) relying on $h$ linearly.
\end{proposition}
\begin{proof}
	[Sketch of the proof]
	We use the same argument for Proposition 5.1 in \cite{ko-scalar-2024} and also track the term $s^3$ to get the following expansions
\begin{align*}
	\cos \measuredangle _{g_0}(\Sigma_{\lambda,s},\partial M)={}&1-\frac{2(\hat{x}_\beta c _{\alpha\beta}b_{\alpha \beta}(1+\lambda))^2}{a_{\alpha\alpha}a^{33}}s^2+A(\hat{x})s^3+O(s^4)\\
	\cos \bar{\gamma}_{\lambda,s}(\hat{x})={}&1-2s^2\left[ (\hat{x}_1c_{11}+\hat{x}_2c_{12})^2+(\hat{x}_1c_{12}+\hat{x}_2c_{22})^2 \right] +A(\hat{x})s^3+O(s^4),
\end{align*}
where we assume $c_{21}=c_{12}$.
Together with the remaining computation for Proposition 5.1 in \cite{ko-scalar-2024} and the Corollary 5.5 in \cite{chai-scalar-2023-arxiv} (see the proof for Corollary 5.17 in \cite{chai-scalar-2023-arxiv}), we can establish the result.
\end{proof}

As a corollary of Proposition \ref{prop_angleAsy}, we can easily establish the following results for $\sin \gamma_s$,
\begin{align}
	\sin \gamma_{s}(\hat{x})={}&\sin \bar{\gamma}_s(\hat{x})+O(s^2)\\
	={}& 2s\sqrt{(\hat{x}_1c_{11}+\hat{x}_2c_{12})^2+(\hat{x}_1c_{12}+\hat{x}_2c_{22})^2}+O(s^2),\label{eq:corAngleSin}
\end{align}
and the following proposition.


\begin{proposition}
	\label{prop_GreaterAngleG0}
	Suppose the conditions in Proposition \ref{prop_angleAsy} hold.
	Then, for any $\lambda>0$, we can find $s_0>0$ (might rely on $\lambda$) such that for any $s<s_0$, we have
	\begin{equation}
		\gamma_{\lambda,s}(\hat{x})>\bar{\gamma}_{\lambda,s}(\hat{x})
		\label{eq:corGreaterAngle}
	\end{equation}
	for any $\hat{x} \in \partial E_\lambda$.
\end{proposition}

We need to analyze the asymptotic behavior of mean curvature.
We define the following mean curvatures:
\begin{align*}
	H_{\lambda,s}^{g}(\hat{x}):={}&\text{Mean curvature of $\Sigma_{\lambda,s}$ at $\Sigma_{\lambda,s}(\hat{x})$ under metric $g$},\\
	H_{\lambda,s,\partial M}^{g}(\hat{x}):={}& \text{Mean curvature of $\partial M$ at $(\varphi(|\Phi_{\lambda,s}(\hat{x})|^2),\Phi_{\lambda,s}(\hat{x}))$ under metric $g$}.
\end{align*}

Using the same computation for Corollary 5.2 in \cite{ko-scalar-2024}, we have
\begin{proposition}
	\label{prop_meanAsy}
	Suppose the metric $g$ can be written as $g=g_0+s h+O(s^2)$ where $g_0$ is a constant metric defined in \eqref{eq:G0Def}, and $h$ is a bounded symmetric two-tensor. 
	Then, we have the following formula for the behavior of mean curvature
	\begin{equation}
		\label{eq:corMean}
		H^g_s(\hat{x})=H^g_{s,\partial M}(\hat{x})-H^{\bar{g}}_{s,\partial M}(\hat{x})-2(c_{11}+c_{22})+\frac{2B}{\sqrt{a_{11}a_{22}a^{33}}}
		+sL(\hat{x})+O(s^2),
	\end{equation}
	for any $\hat{x} \in E$.
	Here, we write $H^g_s=H^g_{0,s}$ and $H^g_{s,\partial M}=H^g_{0,s,\partial M}$ for short.
\end{proposition}

Now, we consider
\[
	H_0:=\lim_{s\rightarrow 0} H_s^g(\hat{x}),
\]
which is well-defined by \eqref{eq:corMean} (the limit does not depend on the choice of $\hat{x}$.)

We have two subcases to consider.

If $H_0<\bar{h}(0)$, then we can use the continuation of $H_{\lambda,s}^g$ with respect to $\lambda$ and $s$, together with Proposition \ref{prop_GreaterAngleG0}, we can show the following results (cf. Proposition 5.10 in \cite{chai-scalar-2023-arxiv}).
\begin{proposition}
	\label{prop_GreaterMean}
	Suppose the metric $g$ can be written as $g=g_0+sh+O(s^2)$ where $g_0$ is a constant metric defined in \eqref{eq:G0Def}, and $h$ is a bounded symmetric two-tensor.
	If $H_0<\bar{h}(0)$, we can choose some $\lambda>0,s>0$ small such that $H_{\lambda,s}^g(\hat{x})>\bar{h}(\Sigma_{\lambda,s}(\hat{x}))$ for any $\hat{x} \in \partial E_\lambda$ and $\gamma_{\lambda,s}(\hat{x})<\bar{\gamma}_{\lambda,s}(\hat{x})$ for each $\hat{x} \in \partial E_\lambda$.
\end{proposition}

Now, we focus on the case $H_0=\bar{h}(0)$.
In particular, it implies $a_{11}=a_{22}=1$ and $H^g_{\partial M}(O)=H^{\bar{g}}_{\partial M}(O)$.

Then, we need to construct a foliation near $O$.
We define the vector field $Y_s(\hat{x}):=\frac{\partial }{\partial s}\Sigma_s(\hat{x})$.
Given $u \in C^{1,\alpha}(\bar{E})\cap C^{2,\alpha}(E)$ where $E=E_0$, we can define the perturbation surface $\Sigma_{s,u}$ by
\[
	\Sigma_{s,u}:=
	\left\{ \Sigma_{s+\frac{u}{\left< Y_s(\hat{x}),N_s(\hat{x}) \right> }}(\hat{x}):\hat{x} \in E \right\}
\]
where $N_s(\hat{x})$ is the unit normal vector field of $\Sigma_s$.


We write $E=E_0$. Using the variational formula for mean curvatures and contact angles, we have
\begin{align*}
	\frac{H_{s,s^3u}-\bar{h}_{s,s^3u}}{s}={} & -\Delta_s^E u+\frac{H_s-\bar{h}_s}{s}+O(s),\\
	\frac{\cos \gamma_{s,s^3u}-\cos \bar{\gamma}_{s,s^3u}}{s^3}={} & 
	-2s\sqrt{(\hat{x}_1c_{11}+\hat{x}_2c_{12})^2+(\hat{x}_1c_{12}+\hat{x}_2c_{22})^2} \frac{\partial u}{\partial \nu_s^E}\\&+(A_{\partial M}(\eta_s,\eta_s)-\cos \gamma_s A(\nu_s,\nu_s)\\
&-\bar{A}_{\partial M}(\bar{\eta}_s,\bar{\eta}_s))u+
\frac{\cos \gamma_s-\cos \bar{\gamma}_s}{s^3}+O(s),
\end{align*}
where $\Delta_s^E$ denotes the Laplacian-Beltrami operator on $E$ under the metric $\frac{1}{s^2}\Sigma^*_s(g)$, and $\nu_s^E$ is the unit normal vector field of $\partial E$ under the metric $\frac{1}{s^2}\Sigma^*_s(g)$.
Here, we have used \eqref{eq:corAngleSin}.

By using the same argument for Proposition 5.27 in \cite{chai-scalar-2023-arxiv}, together with the asymptotic behavior of mean curvature, for each $s \in (0,\varepsilon)$ sufficiently small, we can find $u_s(\cdot)=u(\cdot,s)$ such that the mean curvature $H_{s,s^3u_s}$ is $\bar{h}_{s,s^3u}+s\lambda(s)$ where $\lambda(s)$ is a function only depends on $s$, the contact angle $\gamma_{s,s^3u_s}=\bar{\gamma}_{s,s^3u_s}$, and $u$ satisfies the following
\[
	\lim_{s\rightarrow 0} (u(\hat{x},s)+u(-\hat{x},s))=0
\]
for any $\hat{x} \in E$.
A finer analysis of $\lambda(s)$ will give $\lambda(s)<0$ for $s$ sufficiently small (cf. Proposition 5.28 in \cite{chai-scalar-2023-arxiv}), and it leads to the following.
\begin{proposition}
	\label{prop_meanConvexSphere}
	We can construct a surface $\Sigma_{-}$ near $O$ such that the mean curvature of $\Sigma_{-}$ is not greater than $\bar{h}$ and it has prescribed contact angle $\bar{\gamma}$ with $\partial M$.
\end{proposition}

\begin{proof}
	[Proof of item (\ref{smooth}) of Theorem \ref{rigidity depending on structures}]
	We can use Proposition \ref{prop_GreaterMean} or Proposition \ref{prop_meanConvexSphere} depending on the value of $H_0$ to construct a barrier surface $\Sigma_{-}$ with mean curvature not greater than $\bar{h}$ and prescribed contact angle $\bar{\gamma}$ with $\partial M$.
	Then, we can use item (\ref{item simple}) of Theorem \ref{rigidity depending on structures} to extend the rigidity to all of $M$.
\end{proof}

\

\bibliographystyle{alpha} 
\bibliography{rigidity-rotationally-symmetric-in-warped-products}
\end{document}